\documentclass[12pt]{amsart}
\usepackage{amscd}
\usepackage{mathrsfs}
\usepackage{a4wide}
\usepackage{color}
\usepackage[utf8x]{inputenc}

\usepackage{amsmath,amsthm}
\usepackage{amssymb}
\usepackage{enumerate}
\usepackage{todonotes}
\usepackage{graphicx}
\usepackage{cancel}
\usepackage{xcolor}
\usepackage[arrow, matrix, curve, cmtip]{xy}
\usepackage{tikz}
\usepackage{geometry}
\usepackage[shortlabels]{enumitem}

\newtheorem{theorem}{Theorem}[section]

\newtheorem{claim}[theorem]{Claim}

\newtheorem{corollary}[theorem]{Corollary}

\newtheorem{definitionandproposition}[theorem]{Definition and Proposition}

\newtheorem{lemma}[theorem]{Lemma}

\newtheorem{observation}[theorem]{Observation}

\newtheorem{question}[theorem]{Question}

\theoremstyle{definition}
\newtheorem{definition}[theorem]{Definition}
\newtheorem{remark}[theorem]{Remark}
\newtheorem{example}[theorem]{Example}

\newcommand{\cf}{\mathrm{cf}}

\newcommand{\bb}{\mathbb}

\newcommand{\chm}{\check{\mathrm{H}}}

\date{}
\begin{document}

\title[The Cohomology of the Ordinals I]{The Cohomology of the Ordinals I: \\ Basic Theory and Consistency Results}
\author{Jeffrey Bergfalk}

\address{Centro de Ciencas Matem\'{a}ticas\\
UNAM\\
A.P. 61-3, Xangari, Morelia, Michoac\'{a}n\\
58089, M\'{e}xico}

\email{jeffrey@matmor.unam.mx}

\author{Chris Lambie-Hanson}

\address{Department of Mathematics and Applied Mathematics \\ 
Virginia Commonwealth University \\ 
Richmond, VA 23284 \\ 
United States}

\email{cblambiehanso@vcu.edu}

\thanks{{\it Date}: \today.\newline
{\it 2010 MSC}: 03E10, 55N05, 54B40, 03E04, 03E35.\newline
{\it Key words and phrases.} ordinal, coherent sequence, walks on ordinals, presheaf, constant sheaf, \v{C}ech cohomology, n-coherence, weakly compact cardinal, strongly compact cardinal, incompactness, V=L, square sequence.\newline
{The research of the first author was partly supported by the NSF grant DMS-1600635.}}

\begin{abstract} In this paper, the first in a projected two-part series, we describe an organizing framework for the study of infinitary combinatorics. This framework is \v{C}ech cohomology. We show in particular that the \v{C}ech cohomology groups of the ordinals articulate higher-dimensional generalizations of Todorcevic's walks and coherent sequences techniques, and begin to account for those techniques' ``unreasonable effectiveness'' on $\omega_1$. This discussion occupies the first half of our paper and is written with a general mathematical audience in mind.

We turn in the paper's second half to more properly set-theoretic considerations. We describe a number of consistency results on the cohomology groups of the ordinals which certify their status as a graded family of incompactness principles. We show in particular that nontrivial cohomology groups on the ordinals are in some tension with large cardinals, and are maximally extant in G\"{o}del's model $\mathrm{L}$. We describe forcings to add, then trivialize, nontrivial $n$-cocycles, and conclude with some comparison of these principles with those benchmark incompactness phenomena, the existence of square sequences and failures of stationary reflection.
\end{abstract}
\maketitle

\section{Introduction}\label{zero}

Among the most consequential developments in the study of infinitary combinatorics in recent decades have been the arrival and elaboration of Todorcevic’s method of minimal walks (see 
\cite{pairs}, \cite{Coherentsequences}, \cite{walks}). This method is rich in applications and startlingly elementary in principle. It powerfully consolidates our understanding of the $\mathsf{ZFC}$ combinatorics of $\omega_1$, even as its importance to the analysis of combinatorial assumptions supplementary to the $\mathsf{ZFC}$ axioms on cardinals above $\omega_1$ continues to grow.

In the following, we describe an organizing framework for the study of infinitary combinatorics. This framework is \v{C}ech cohomology. We show in particular that the \v{C}ech cohomology groups of the ordinals articulate higher-dimensional generalizations of Todorcevic's walks and coherent sequences techniques, and begin to account for those techniques' ``unreasonable effectiveness'' on $\omega_1$.

This work is the first in a projected two-part series. Its aim is to introduce the aforementioned groups and to describe a number of consistency results thereon. By facts discussed in Section \ref{I.5} below, such results are essentially only possible for groups $\chm^n(\xi)$ in which $n$ is positive and the cofinality of $\xi$ is greater than $\omega_n$. In contrast, the groups $\chm^n(\omega_n)$ each articulate nontrivial $(n+1)$-dimensional combinatorial relations first appearing at the ordinal $\omega_n$ in any model of the $\mathsf{ZFC}$ axioms; as suggested, the case of $n=1$ is simply, at present, the best understood. These groups $\chm^n(\omega_n)$ form the main subject of the forthcoming \textit{The Cohomology of the Ordinals II: ZFC Results} \cite{CohomologyII}.

The plan of the present paper is the following:
\begin{itemize}
\item In Section \ref{O.1} we record our main notational conventions.
\item In Section \ref{I.1} we review Todorcevic's method of minimal walks.
\item In Section \ref{I.2} we review the definition of \v{C}ech cohomology with respect to a presheaf $\mathcal{P}$ of abelian groups. We show that the material of Section \ref{I.1} witnesses that $\chm^1(\omega_1,\mathcal{P})\neq 0$ in a wide variety of cases.
\item In Section \ref{I.3} we define higher-dimensional generalizations of classical nontrivial coherence.
\item In Section \ref{I.4} we show that the non-$n$-trivial $n$-coherent families of Section \ref{I.3} correspond precisely to the higher-dimensional \v{C}ech cohomology groups $\chm^n(\xi,\mathcal{A}_d)$ of any given ordinal $\xi$. Here $\mathcal{A}_d$ denotes the sheaf of locally constant functions to an abelian group $A$.
\item In Section \ref{I.5} we describe several $\mathsf{ZFC}$ theorems constraining the possibilities for this paper's \emph{Section \ref{sectiontwo}: Consistency Results}.
\end{itemize}
A main aim of the above sequence is to render the \v{C}ech cohomology groups of the ordinals more amenable to set-theoretic study. In subsequent sections, we turn more directly to that study itself; the groups described below are in general computed with respect to some arbitrary nontrivial constant abelian sheaf:
\begin{itemize}
\item In Section \ref{II.1} we show that the \v{C}ech cohomology groups $\chm^n(\lambda)$ are trivial for $n\geq 1$ when $\lambda$ is a weakly compact cardinal. Similarly for any $\lambda$ greater than or equal to a strongly compact cardinal $\kappa$.
\item In Section \ref{II.2} we describe two scenarios --- one via the L\'{e}vy collapse of a strongly compact cardinal, and one, due to Todorcevic, via the P-Ideal Dichotomy --- for the vanishing of all groups $\chm^1(\lambda)$ in which $\lambda$ is a regular cardinal greater than $\aleph_1$.
\item In Section \ref{II.3} we show in contrast that in G\"{o}del's constructible universe $\mathrm{L}$, the \v{C}ech cohomology groups of the ordinals are nontrivial anywhere they can be (this latter phrase has a precise and extensive meaning grounded in the results of Section \ref{I.5} and \ref{II.1}).
\item In Section \ref{II.4} we describe a family of forcing notions for adding non-$n$-trivial $n$-coherent families of functions of height $\lambda$ (and, hence, for adding nontrivial cohomology groups $\chm^n(\lambda)$). We describe also related forcing notions for $n$-trivializing those families in turn.
\item In Section \ref{II.5} we apply results of Section \ref{II.4} in several arguments distinguishing nontrivial coherence from several other phenomena very close in spirit --- namely, the existence of square sequences and failures of stationary reflection.
\end{itemize}
Conspicuous in the above sequence is the set-theoretic theme of \emph{incompactness}, i.e., of phenomena at a cardinal $\lambda$ abruptly unlike phenomena below (see \cite{compactness}). Such phenomena generally signal some affinity for G\"{o}del's model $\mathrm{L}$ and some tension with large cardinal principles, criteria the existence of non-$n$-trivial $n$-coherent families of functions very plainly fulfills. From this perspective, the \v{C}ech cohomology groups of the ordinals constitute a graded family of incompactness principles whose interrelations we are only just beginning to understand. Most immediately, it remains an open question under what conditions the group $\chm^2(\aleph_3)$, for example, may vanish, and it is with this and related questions that we conclude.

\subsection{Notations, conventions, and references}\label{O.1}
We describe first the sorts of knowledge presumed of our readers. In Sections \ref{zero} and \ref{sectionone} we assume little more than a basic familiarity with the ordinals and with cochain complexes; the further facts we require about stationary sets are reviewed at the end of this subsection. Readers unfamiliar with chain homotopies or with the derivation of a long exact sequence from a short one are referred to \cite[Sections 1.4 and 1.3, respectively]{weibel}. We invoke but touch only lightly on the subject of sheafification (similarly for the relations between \v{C}ech and sheaf cohomology); readers looking for more are referred to the standard references \cite{bredonsheaves}, \cite{Godement}, \cite{Tohoku}, or \cite{Hartshorne}. Readers looking for more on the subjects of walks on the ordinals and coherent sequences are referred to \cite{pairs}, \cite{Coherentsequences}, and \cite{walks}; valuable prior recognitions of set-theoretic combinatorial phenomena as cohomological in nature are \cite{blass}, \cite{talaycoi}, and \cite{talaycoii}. We should perhaps reiterate, though, that Sections \ref{zero} and \ref{sectionone} are, by design, largely self-contained.

In Section \ref{sectiontwo} we assume rather more set-theoretic knowledge of our reader. We follow broadly the conventions of \cite{kunen} and \cite{kanamori}; several points, though, merit particular mention or emphasis:
\begin{itemize}
\item The first lemma of Section \ref{sectiontwo} assumes some knowledge of the categories $\mathsf{inv}$-$\mathrm{Ab}$, $\mathsf{pro}$-$\mathrm{Ab}$, and of the functoriality of $\mathrm{lim}^n$ over either. Readers are referred to \cite[Chapters 1 and 15]{strongshape} for an accessible account. However, this machinery may be safely ignored; it is not invoked elsewhere and readers so inclined may argue the lemma by elementary means.
\item We write $x =^* y$ if the symmetric difference of two sets (or of two functions viewed as sets) $x$ and $y$ is finite.
\item We write $\varphi\!\restriction\!X$ for the restriction of the function $\varphi$ to the domain $X$.
\item We write $\text{cf}(\xi)$ for the cofinality of $\xi$, that is, for the minimal order-type of a cofinal subset of $\xi$. We write $\text{otp}(A)$ for the order-type of a set $A$ of ordinals. We write $\text{Cof}(\xi)$ for the class of ordinals of cofinality $\xi$.
\item We are somewhat lax in notationally distinguishing between the roles of any cardinal as an ordinal and as a cardinal; our main concern is for readability.
\item We write $[A]^n$ for the family of $n$-element subsets of any collection of ordinals $A$, and frequently identify such subsets with their natural orderings
$\vec{\alpha}=(\alpha_0,\dots,\alpha_{n-1})$. For $j < n$, $\vec{\alpha}^j$ denotes that element of $[A]^{n-1}$ obtained via deletion of the $j$-indexed coordinate of $\vec{\alpha}$, i.e., $\vec{\alpha}^j = (\alpha_0,\dots,\alpha_{j-1},\alpha_{j+1},\dots,\alpha_{n-1})$. We often notate these tuples as concatenations when they are in the position of subscripts, writing $\varphi_{\alpha\beta}$ instead of $\varphi_{(\alpha, \beta)}$ or $\varphi_{\vec{\alpha}\beta}$ instead of $\varphi_{\vec{\alpha}^\frown (\beta)}$, for example. Though we have imperfectly succeeded, we have aimed in general for alphabetical orderings of symbols consistent with their intended values.
\item Topological considerations are always with respect to the order topology on an ordinal $\xi$, i.e., with respect to the topology generated by the initial and the terminal segments of $\xi$.\footnote{Higher derived limits are sometimes cast as \v{C}ech cohomology groups of their index-sets endowed with the \emph{initial segments} topology (see \cite{jensencu}); a distinction of the present work is its focus on the ordinal topologies of set-theoretic practice.}
\end{itemize}
It is with respect to this topology that one of the most essential of set-theoretic notions is defined: the closed unbounded subsets of an ordinal, which are typically termed \textit{clubs}. 
Subsets of $\beta$ intersecting every club subset of $\beta$ are termed \textit{stationary}. Each of these classes of subsets carries connotations of largeness. For example, the club subsets of an ordinal $\beta$ of uncountable cofinality form a filterbase on $\beta$. A measure analogy is then natural, in which $\beta$'s club filter collects the ``measure one'' subsets of $\beta$, and the stationary subsets of $\beta$ are in consequence the ``sets of positive outer measure.'' 
\begin{lemma}[The Pressing Down Lemma (Alexandroff, 1929)]\label{pdl}
 Let $\kappa$ be a regular uncountable cardinal. Let $S\subseteq\kappa$ be stationary and let $f:S\rightarrow\kappa$ be be such that $f(\alpha)<\alpha$ for all $\alpha\in S$. Then $f$ is constant on some stationary $T\subseteq S$.
\end{lemma}
We will apply this lemma repeatedly below, by way of the following corollary.
\begin{corollary}\label{thesubcover} Let $S$ be a stationary subset of a regular uncountable cardinal $\kappa$. Then for any family $\mathcal{U}$ of bounded open subintervals of $\kappa$ such that $\cup\,\mathcal{U}\supseteq S$ there exist an $\alpha < \kappa$, a stationary $T\subseteq S$, and a collection $\mathcal{V}=\{(\alpha,\beta_\xi)\,|\,\xi\in T\}\subseteq\mathcal{U}$ with the property that $\xi\in(\alpha,\beta_\xi)$ for all $\xi\in T$.
\end{corollary}
\begin{proof} For each $\xi\in S$ let $(\alpha_\xi,\beta_\xi)\in\mathcal{U}$ be such that $\xi\in(\alpha_\xi,\beta_\xi)$. Apply Lemma \ref{pdl} to the function $f:\xi\mapsto\alpha_\xi$.
\end{proof}
In particular, no stationary $S\subseteq\kappa$ is paracompact in the subspace topology. Hence no stationary subspace $S\subseteq\kappa$ is metrizable. (The metrizable subspaces of $\omega_1$, in fact, are precisely the nonstationary $N\subseteq\omega_1$; in the measure analogy, they are the ``vanishingly small'' subsets of $\omega_1$.) This underscores the oddness of our results below: techniques developed expressly for the study of polyhedra and manifolds, techniques moreover in which connectedness assumptions figure centrally, will at times appear tailor-made for the study of a family of spaces at once highly discrete and far from paracompact. The family we mean, of course, is that of the ordinals $\beta$ of uncountable cofinality.

\section{Coherence and Cohomology}\label{sectionone}

\subsection{Walks and coherence}\label{I.1}

For reasons that should soon be clear, we begin our account with the eponymous \textit{coherent sequences} of \cite{Coherentsequences}. This subsection is an inevitably incomplete overview of the material of that work together with that of \cite{pairs} and \cite{walks}; this body of results will form our guiding example of what a nontrivial \v{C}ech cohomology group of an ordinal may signify.

\begin{definition}
Let $\delta$ be an ordinal. A \emph{C-sequence on $\delta$} is a family $\langle\,C_\alpha\,|\,\alpha\in\delta\,\rangle$ in which each $C_\alpha$ is a cofinal subset of $\alpha$.
\end{definition}

Any C-sequence on $\delta$ determines a system of \textit{walks} on the pairs of ordinals in $\delta$. Such a walk from any $\beta$ in $\delta$ down to any lower $\alpha$ is conceived of as a finite decreasing sequence of \textit{steps}
$$\big[\,\beta_0=\beta\,\big]>\dots>\big[\,\beta_i=\min(C_{\beta_{i-1}}\backslash \alpha)\,\big]>\dots>\big[\,\beta_k=\min(C_{\beta_{k-1}}\backslash \alpha)=\alpha\,\big]$$
The collection of such steps is denoted $\text{Tr}(\alpha,\beta)$. In definitions such as the following, assume some C-sequence on $\delta$ to have been fixed.

\begin{definition}\label{uppertrace} The \emph{upper trace function} $\textnormal{Tr}:[\delta]^2\rightarrow[\delta]^{<\omega}$ is recursively defined as follows:
$$\textnormal{Tr}(\alpha,\beta)=\{\beta\}\cup\textnormal{Tr}(\alpha,\textnormal{min}(C_\beta\backslash\alpha))$$
where $\textnormal{Tr}(\alpha,\alpha)=\{\alpha\}$ for each $\alpha\in\delta$.
\end{definition}

The core constructions in this subject are, like Tr$(\,\cdot\,,\,\cdot\,)$, all recursive on the input of the C-sequence. Stronger conditions on that sequence tend accordingly to manifest as more informative or intricate combinatorial outputs. Articulating much of this information are Todorcevic's \textit{rho functions}, the so-called \textit{statistics} or \textit{characteristics} of the walks.  Two of the more natural and prominent such characteristics are a walk's ``height'' and ``length,'' the functions $\rho_1$ and $\rho_2$, respectively, of the following definition:
\begin{definition} The \emph{maximum weight} and \emph{number of steps}\footnote{Strictly speaking, this is a simplifying abuse: $\rho_2(\alpha,\beta)$ historically equals $|\textnormal{Tr}(\alpha,\beta)|-1$. The abuse simply streamlines the proof of Claim \ref{rho2claim} below.} functions $[\delta]^2\rightarrow\delta$ are, respectively,
\begin{itemize}
\item $\rho_1(\alpha,\beta)=\textnormal{max}\{\textnormal{otp}(C_\xi\cap\alpha)\,|\:\xi\in \textnormal{Tr}(\alpha,\beta)\backslash\{\alpha\}\,\}$
\item $\rho_2(\alpha,\beta)=|\textnormal{Tr}(\alpha,\beta)|$
\end{itemize}
\end{definition}

A certain rapport between these functions will be thematic below. Observe firstly, for example, that if a C-sequence on $\delta=\omega_1$ satisfies
\begin{align*}
\tag{$\star$} \text{otp}(C_\alpha)=\text{cf}(\alpha)\hspace{.18 cm}\text{for all }\alpha\in\delta
\end{align*}
then $\rho_1(\,\cdot\,,\,\cdot\,)$ and $\rho_2(\,\cdot\,,\,\cdot\,)$ are each integer-valued. In fact in this case, much deeper affinities are apparent:
\begin{theorem}[\cite{pairs}]\label{coherenceofrhos} Fix a C-sequence on $\delta=\omega_1$ satisfying $(\star)$. Then
\begin{align}\label{coh1}\rho_1(\,\cdot\,,\gamma)\!\restriction\!\beta\;=^*\rho_1(\,\cdot\,,\beta)\textnormal{ for all }\beta<\gamma<\omega_1,\end{align}
but there exists no $\tilde{\rho}_1:\omega_1\rightarrow\mathbb{Z}$ such that
\begin{align}\label{triv1}\tilde{\rho}_1(\,\cdot\,)\!\restriction\!\beta\;=^*\rho_1(\,\cdot\,,\beta)\textnormal{ for all }\beta<\omega_1.\end{align}
Similarly, writing $=^b$ for ``equality modulo a bounded function,''
\begin{align*}\tag{1'}\label{coh2}\rho_2(\,\cdot\,,\gamma)\!\restriction\!\beta\;=^b\rho_2(\,\cdot\,,\beta)\textnormal{ for all }\beta<\gamma<\omega_1,\end{align*}
but there exists no $\tilde{\rho}_2:\omega_1\rightarrow\mathbb{Z}$ such that
\begin{align*}\tag{2'}\label{triv2}\tilde{\rho}_2(\,\cdot\,)\!\restriction\!\beta\;=^b\rho_2(\,\cdot\,,\beta)\textnormal{ for all }\beta<\omega_1.\end{align*}
\end{theorem}

In short, for any C-sequence on $\omega_1$ satisfying $(\star)$, each rho function exhibits relations of local agreement (1, 1') that cannot be globalized (2, 2'). The term \emph{nontrivial coherence} refers broadly to relations of this sort; we will assign it more concrete meaning below. Such relations are our main focus in this and the following subsection; to that end, until otherwise indicated, let $\delta=\omega_1$ and fix a C-sequence on $\delta$ as above.

Useful in arguing Theorem \ref{coherenceofrhos} will be the following function.

\begin{definition} Adopt the convention that $\max (\varnothing)=0$. The \emph{maximum of the lower trace function} $max\: L:[\omega_1]^2\rightarrow\omega_1$ is defined as follows:
$$max\: L(\alpha,\beta)=
\textnormal{max}\{\max(C_\xi\cap\alpha)\,|\:\xi\in \textnormal{Tr}(\alpha,\beta)\backslash\{\alpha\}\,\}$$
\end{definition}

The condition $(\star)$ ensures that this function is well-defined. The function's utility is the following: for any $\alpha<\beta<\gamma$,
$$\big[max\: L(\beta,\gamma)<\alpha\big]\Rightarrow\big[\beta\in\textnormal{Tr}(\alpha,\gamma)\big]\Rightarrow\big[\textnormal{Tr}(\alpha,\gamma)=\textnormal{Tr}(\alpha,\beta)\cup\textnormal{Tr}(\beta,\gamma)\big],$$
as the reader may verify.

\begin{proof}[Proof of Theorem \ref{coherenceofrhos}] We argue only (\ref{coh2}) and (\ref{triv2}). The proofs of  (\ref{coh1}) and (\ref{triv1}) are similar in spirit and simpler in their details (cf.\ \cite{MooreLSpace}).

\textit{The coherence of $\rho_2$}:
Suppose towards contradiction that for some $\beta<\gamma$ in $\omega_1$, $$|\rho_2(\alpha_k,\gamma)-\rho_2(\alpha_k,\beta)|>k$$ for each $\alpha_k$ in some increasing sequence of ordinals $\langle \alpha_k\,|\,k\in\omega\rangle\subseteq\beta$. Let $\alpha=\sup_{k\in\omega}\alpha_k$. Let $\eta=\max\{ max\: L(\alpha,\beta),max\: L(\alpha,\gamma)\}$. Observe that for $\alpha_k>\eta$,
\begin{align*}
\textnormal{Tr}(\alpha_k,\beta) & =\textnormal{Tr}(\alpha_k,\alpha)\cup\textnormal{Tr}(\alpha,\beta) \\
\textnormal{Tr}(\alpha_k,\gamma) & =\textnormal{Tr}(\alpha_k,\alpha)\cup\textnormal{Tr}(\alpha,\gamma)
\end{align*}
In consequence, $\alpha_k>\eta$ implies that $|\rho_2(\alpha_k,\gamma)-\rho_2(\alpha_k,\beta)|=|\rho_2(\alpha,\gamma)-\rho_2(\alpha,\beta)|$. Hence any $\alpha_k>\eta$ with $k\geq |\rho_2(\alpha,\gamma)-\rho_2(\alpha,\beta)|$ will contradict our premise.

\textit{The nontriviality of $\rho_2$}: Suppose towards contradiction that for some $\tilde{\rho}_2:\omega_1\rightarrow\mathbb{Z}$, for all $\gamma\in\omega_1$ there exists a $k(\gamma)\in\mathbb{N}$ such that
$$|\tilde{\rho}_2(\beta)-\rho_2(\beta,\gamma)|\leq k(\gamma)\text{ for all }\beta<\gamma\;$$
The function $\tilde{\rho}_2$ constantly equals $m$ on some cofinal $B\subseteq\omega_1$, and the function $k$ constantly equals $n$ on some cofinal $G\subseteq\omega_1$. The following claim will furnish the desired contradiction.
\begin{claim}\label{rho2claim} Let $\mathcal{A}=\{(\beta_i,\gamma_i)\,|\,i\in\omega_1\}\subseteq \omega_1\times\omega_1$ satisfy $\max\{\beta_i,\gamma_i\}<\min\{\beta_j,\gamma_j\}$ for every $i<j$ in $\omega_1$. Then for any $\ell\in\mathbb{N}$ there exists a cofinal $\Gamma\subseteq\omega_1$ such that $\rho_2(\beta_i,\gamma_j)>\ell$ for any $i<j$ in $\Gamma$.
\end{claim}

Taking an $\mathcal{A}\subseteq B\times G$ as in the claim, we may conclude that there exists a $\beta\in B$ and $\gamma\in G\backslash\beta$ with $\rho_2(\beta,\gamma)>m+n$. In consequence, 
$$n=k(\gamma)\geq |\tilde{\rho}_2(\beta)-\rho_2(\beta,\gamma)|>|m-(m+n)|=n$$
This is a contradiction.
\begin{proof}[Proof of Claim \ref{rho2claim}] We argue by induction on $\ell\in\mathbb{N}$. The case $\ell=0$ is clear. Therefore suppose that the claim holds for some $\ell\in\mathbb{N}$. We show that it holds for $\ell+1$ as well:

For each limit ordinal $\xi\in\omega_1$ let $i(\xi)$ be the least $i$ such that $\beta_i$ and $\gamma_i$ are greater than $\xi$. The associated function $f:\xi\mapsto max\:L(\xi,\gamma_{i(\xi)})$ is regressive; hence $f''S=\{\eta\}$ for some $\eta\in\omega_1$ and stationary $S\subseteq\omega_1$. Thin $S$ if necessary to some cofinal $R\subseteq S$ such that $\beta_{i(\xi)}<\zeta<\gamma_{i(\zeta)}$ for all $\xi<\zeta$ in $R$. Now apply the induction hypothesis to the family $\{(\beta_{i(\xi)},\xi)\,|\,\xi\in R\}$ to find a cofinal $P\subseteq R$ such that $\rho_2(\beta_{i(\xi)},\zeta)>\ell$ whenever $\xi<\zeta$ are in $P$. By definition, for all $\xi<\zeta$ in $R$,
$$\textnormal{Tr}(\beta_{i(\xi)},\gamma_{i(\zeta)})=\textnormal{Tr}(\beta_{i(\xi)},\zeta)\cup\textnormal{Tr}(\zeta,\gamma_{i(\zeta)})$$Hence $\rho_2(\beta_{i(\xi)},\gamma_{i(\zeta)})>\ell+1$ for any $\xi<\zeta$ in $P$, and letting $\Gamma=\{i(\xi)\,|\,\xi\in P\}$ completes the induction step.
\end{proof}

\end{proof}
Significant below will be the following less-remarked variant of (\ref{coh2}) and (\ref{triv2}).
\begin{corollary}\label{locallyconstantcoherence}Let  $=^c$ denote ``equality modulo a locally constant function.'' Then
\begin{align*}\tag{1''}\label{coh3}\rho_2(\,\cdot\,,\gamma)\!\restriction\!\beta\;=^c\rho_2(\,\cdot\,,\beta)\textnormal{ for all }\beta<\gamma<\omega_1,\end{align*}
but there exists no $\tilde{\rho}_2:\omega_1\rightarrow\mathbb{Z}$ such that
\begin{align*}\tag{2''}\label{triv3}\tilde{\rho}_2(\,\cdot\,)\!\restriction\!\beta\;=^c\rho_2(\,\cdot\,,\beta)\textnormal{ for all }\beta<\omega_1.\end{align*}
\end{corollary}
\begin{proof} By the above argument for the ($=^b$) coherence of $\rho_2$, for every $\beta<\gamma$ in $\omega_1$ and limit ordinal $\alpha\in\beta$,
$$\rho_2(\,\cdot\,,\gamma)\!\restriction\!\beta\,-\rho_2(\,\cdot\,,\beta)$$ is constant on some nontrivial interval $(\eta,\alpha)$. This shows (\ref{coh3}).

To see (\ref{triv3}), assume towards contradiction that for some $\tilde{\rho}_2:\omega_1\rightarrow\mathbb{Z}$,
\begin{align*}\tilde{\rho}_2(\,\cdot\,)\!\restriction\!\beta+1\;=^c\rho_2(\,\cdot\,,\beta+1)\textnormal{ for all }\beta<\omega_1\end{align*}
for all limit ordinals $\beta\in\omega_1$. Then there exists some $\eta\in\omega_1$ and stationary $S$ such that the functions
\begin{align*}\big(\tilde{\rho}_2(\,\cdot\,)-\rho_2(\,\cdot\,,\beta+1)\big)\!\!\restriction\![\eta,\beta]\hspace{.7 cm}(\beta\in S)\end{align*}
are each constant, and therefore bounded. From this family of functions, a contradiction follows exactly as in the proof of Theorem \ref{coherenceofrhos}.
\end{proof}
We close this subsection with a representative corollary and several concluding remarks.
\begin{corollary}\label{aronszajn} For $i=1$ or $2$, the tree $T(\rho_i):=\{\rho_i(\,\cdot\,,\beta)\!\restriction\!\alpha\:|\,\alpha\leq\beta<\omega_1\}$, ordered by extension, is an Aronszajn tree.
\end{corollary}
\begin{proof}
Coherence in each case ensures that the levels of the tree are countable, while nontriviality precludes the existence of a cofinal branch.
\end{proof}

\begin{remark}\label{remark} Four summary observations are at this point in order:
\begin{enumerate}
\item[(i)] The above sequence, culminating 
in the succinct construction of a canonical combinatorial object on $\omega_1$, was representative in the following strong sense: by way of the induced nontrivial coherent families of functions, the walks technique ``can be used to derive virtually all known other structures that have been defined so far on $\omega_1$'' (\cite{walks}, p. 19). With only a little more work, for example, we might equally have concluded with the construction of an uncountable linear order whose square is a union of countably many chains, or of strong witnesses to negative partition relations on $[\omega_1]^2$.
\item[(ii)] The combinatorics of the above-mentioned structures tend overwhelmingly to be ``two dimensional''; they tend in other words to fundamentally involve two-coordinate arguments. Underlying these structures, as we have just seen, are what might themselves be regarded as \textit{rigidity of dimension} phenomena: the nontrivial coherence relations of Theorem \ref{coherenceofrhos}. For \textit{coherence} names relations in which the second coordinate does not matter much, while \textit{nontriviality} names relations in which the second coordinate, nevertheless, cannot be done without.
\item[(iii)] Relatedly, Theorem \ref{coherenceofrhos} and Corollary \ref{locallyconstantcoherence} hint at the variety of forms of nontrivial coherence obtaining on $\omega_1$.

\item[(iv)] Clearly, generalizations of so versatile a technique as we have described to ordinals above $\omega_1$ would be of great value. However,\\

 \begin{quote} \textit{An interesting phenomenon that one realizes while analyzing walks on ordinals is the special role of the first uncountable ordinal $\omega_1$ in this theory. [...] The first uncountable cardinal is the only cardinal on which the theory can be carried out without relying on additional axioms of set theory.} (\cite{walks}, p. 14)\\ \end{quote}
 
\noindent For example, when $\delta=\omega_1$, a limited invocation of the axiom of choice --- the condition $(\star)$ above --- induces nontrivial coherence relations of the sorts described in Theorem \ref{coherenceofrhos} and Corollary \ref{locallyconstantcoherence}. As we will see below, the existence of such relations on an ordinal $\delta$ of any other cofinality is either independent of, or outright inconsistent with, the ZFC axioms.
\end{enumerate}
\end{remark}
\subsection{Coherence and \v{C}ech cohomology}\label{I.2}

Cohomological techniques arguably offer a uniquely cohesive approach to the above family of phenomena. The plurality of forms evoked in item (iii), for example, will manifest naturally as parameters in sheaf-theoretic framings of this material. We now describe these framings. For reasons that should soon be clear, we begin by fixing a primary mathematical meaning for \textit{nontrivial coherence}.

\begin{definition}\label{fixntc} Let $A$ be an abelian group and let $\Phi = \{\varphi_\beta:\beta \rightarrow A\,|\,\beta\in \delta\}$ 
be a family of functions.
\begin{itemize}
	\item $\Phi$ is \emph{coherent} if 
		\begin{align*}
			\varphi_\gamma\!\restriction\!\beta\,=^*\varphi_\beta
		\end{align*}
		for all $\beta\leq\gamma$ in $\delta$.
	\item $\Phi$ is \emph{trivial} if, for some $\varphi:\delta \rightarrow A$,
		\begin{align*}
			\varphi\!\restriction\!\beta\,=^*\varphi_\beta
		\end{align*} 
		for all $\beta$ in $\delta$.
\end{itemize}
\end{definition}
By the identification of any ordinal with its set of predecessors, such a family $\Phi$ may be regarded as indexed by the collection of \textit{open sets} $\{\beta\,|\,\beta\in\delta\}$. When $\delta$ is a limit ordinal, this collection is an open cover of $\delta$. Presheaves are natural frameworks for the study of families of this sort.

\begin{definition}
\emph{A presheaf of abelian groups $\mathcal{P}$ on a topological space $X$} is a contravariant functor from $\tau(X)$ to the category of abelian groups. It is, in other words, an assignment of a group $\mathcal{P}(U)$ to each $U\in\tau(X)$, together with homomorphisms $p_{UV}:\mathcal{P}(U)\rightarrow\mathcal{P}(V)$ for each $U\supseteq V$ in $\tau(X)$, such that $p_{UU}=\text{id}$ and $p_{VW}\circ p_{UV}=p_{UW}$ for all $U\supseteq V\supseteq W$ in $\tau(X)$. It is convenient to require that $\mathcal{P}(\varnothing)=0$ as well.
\end{definition}

Elements of $\mathcal{P}(U)$ are termed \textit{sections of $\mathcal{P}$ over $U$}. The maps $p_{UV}$ are termed, and sometimes written as, \textit{restriction maps}, even when they are mathematically something else. These maps will be ``honest'' restriction maps in all but one of our examples, which we therefore define simply by way of the assignments $U\mapsto\mathcal{P}(U)$.

\begin{example} For any space $X$ and group $A$, the functor $\mathcal{D}_A:U\mapsto \bigoplus_U A$ is a presheaf.
\end{example}

\begin{definition}\label{sheef}
\emph{A sheaf on $X$} is a presheaf $\mathcal{S}$ on $X$ with the following additional property: for any $\mathcal{V}\subseteq\tau(X)$, if $\{s_V\in \mathcal{S}(V)\,|\,V\in\mathcal{V}\}$ are such that $$s_V\!\restriction\!V\cap W\,=s_W\!\restriction\!V\cap W\text{ for all }V,W\in\mathcal{V}$$then there exists a unique $s\in\mathcal{S}(\cup\mathcal{V})$ such that $s\!\restriction\!V\,=s_V$ for all $V\in\mathcal{V}$.
\end{definition}

\begin{example} $\mathcal{D}_A$ is not a sheaf on any infinite Hausdorff space $X$. Neither is the \emph{constant presheaf} $\mathcal{C}_A:U\mapsto A$ $($with restriction maps the identity$)$ if $\tau(X)$ contains disjoint open sets.
\end{example}

The morphisms between two presheaves on a space $X$ are the natural transformations between those presheaves viewed as functors. In this way, the presheaves of abelian groups on a fixed space $X$ define a category; a universal property therein associates to each presheaf $\mathcal{P}$ some ``nearest'' presheaf $\mathtt{sh}(\mathcal{P})$ satisfying the conditions of Definition \ref{sheef}. We follow \cite{Hartshorne} in the definition and proposition below, in which all (pre)sheaves under discussion are on some fixed space $X$:

\begin{definitionandproposition} For any presheaf $\mathcal{P}$ there exists a sheaf $\mathtt{sh}(\mathcal{P})$ and a morphism $p:\mathcal{P}\to\mathtt{sh}(\mathcal{P})$ such that for any sheaf $\mathcal{S}$ and morphism $s:\mathcal{P}\to\mathcal{S}$, there exists a unique $r:\mathtt{sh}(\mathcal{P})\to\mathcal{S}$ such that $s=rp$.
\end{definitionandproposition}

See any standard reference, or Definition \ref{defprop2} below, for explicit constructions; often this optimal modification $\mathtt{sh}(\mathcal{P})$ of a presheaf $\mathcal{P}$ amounts simply to a relaxation of conditions on sections to local ones. The \textit{constant sheaf} $\mathtt{sh}(\mathcal{C}_A)$ is a fundamental example.

\begin{example} Write $\mathcal{A}_d$ for $\mathtt{sh}(\mathcal{C}_A)$, the sheaf $$U\mapsto\{f:U\rightarrow A\,|\,f\textnormal{ is \textit{locally} constant}\}.$$ More generally: for $A$ a topological abelian group, write $\mathcal{A}$ for the sheaf $$U\mapsto\{f:U\rightarrow A\,|\,f\textnormal{ is continuous}\}.$$ $\mathcal{A}_d$, of course, is a special case of the latter: the subscript ``d'' signals a discrete topology on $A$. Similarly, $$U\mapsto\{f:U\rightarrow A\,|\,f\textnormal{ is \textit{locally} finitely supported}\}$$is the sheafification $\mathtt{sh}(\mathcal{D}_A)$ of $\mathcal{D}_A$.
\end{example}

All of our examples will be \textit{presheaves of functions to $A$}, i.e., $\mathcal{P}(U)$ always will be some subcollection of the functions from $U$ to $A$, added pointwise.

Any presheaf $\mathcal{P}$ on a space $X$ induces a series of cohomology groups; prominent (and, typically, most concrete) among these are the \v{C}ech cohomology groups $\check{\mathrm{H}}^n(X,\mathcal{P})$.

\begin{definition}\label{chc} Fix $\mathcal{V}=\{V_\alpha\,|\,\alpha\in\delta\}$, an open cover of $X$. Write $\mathrm{H}^n(\mathcal{V},\mathcal{P})$ for the $n^{th}$ cohomology group of the cochain complex \begin{align}\label{firstcochain}
\mathcal{L}(\mathcal{V},\mathcal{P}):\hspace{.6 cm}0\rightarrow L^0(\mathcal{V},\mathcal{P})\rightarrow\cdots\rightarrow L^j(\mathcal{V},\mathcal{P})\xrightarrow{d^j} L^{j+1}(\mathcal{V},\mathcal{P})\rightarrow\dots
\end{align}
Here $$L^j(\mathcal{V},\mathcal{P})=\prod_{\vec{\alpha}\in [\delta]^{j+1}}\mathcal{P}(V_{\vec{\alpha}}),$$
where $V_{\vec{\alpha}}=V_{\alpha_0}\cap\dots\cap V_{\alpha_{j-1}}$. Write then $p_{\vec{\alpha}\vec{\beta}}$ for $p_{V_{\vec{\alpha}}V_{\vec{\beta}}}$, and define $d^j:L^j(\mathcal{V},\mathcal{P})\rightarrow L^{j+1}(\mathcal{V},\mathcal{P})$ by
$$d^j\!f\,:\,\vec{\alpha}\,\mapsto\,\sum_{i=0}^{j+1}(-1)^i p_{\vec{\alpha}^i\vec{\alpha}}(f(\vec{\alpha}^i))$$
Write $\mathcal{V}\leq\mathcal{W}$ if the open cover $\mathcal{W}$ refines $\mathcal{V}$, i.e., if there exists some $r_{\mathcal{W}\mathcal{V}}:\mathcal{W}\rightarrow\mathcal{V}$ such that $W\subseteq r_{\mathcal{W}\mathcal{V}}(W)$ for each $W\in\mathcal{W}$. The induced $r^{*}_{\mathcal{W}\mathcal{V}}:\mathrm{H}^n(\mathcal{V},\mathcal{P})\rightarrow\mathrm{H}^n(\mathcal{W},\mathcal{P})$ is independent of the choice of refining map $r_{\mathcal{W}\mathcal{V}}$. Hence these $r^{*}_{\mathcal{W}\mathcal{V}}$ $(\mathcal{V}\leq\mathcal{W})$ define, in turn, a direct limit
\begin{align}
\check{\mathrm{H}}^n(X,\mathcal{P}):=\varinjlim_{\mathcal{V}\in\text{Cov}(X)}\mathrm{H}^n(\mathcal{V},\mathcal{P})
\end{align}
This limit is the $n^{th}$ \textit{\v{C}ech cohomology group of $X$, with respect to the presheaf $\mathcal{P}$}.
\end{definition}

\begin{example}\label{zerogroup} $\check{\mathrm{H}}^0(\omega,\mathcal{D}_A)=\prod_\omega A$. This is easy to see: ordered by refinement, the open covers of $\omega$ have a maximum, namely $\mathcal{V}_\omega=\{\{i\}\,|\,i\in\omega\}$. Hence $\check{\mathrm{H}}^0(\omega,\mathcal{D}_A)=\mathrm{H}^0(\mathcal{V}_\omega,\mathcal{D}_A)=\ker(d^0)=L^0(\mathcal{V}_\omega,\mathcal{D}_A)$ $($recall that $\mathcal{D}_A(\varnothing)=0)$, which is, clearly, $\prod_\omega A$. In contrast, $\check{\mathrm{H}}^0(\delta,\mathcal{D}_A)=\bigoplus_\delta A$, for an ordinal $\delta$ of any cofinality other than $\omega$. This is essentially because for any basic open cover $\mathcal{V}$ of $\delta$ and disjoint $\{V_i\,|\,i\in\omega\}\subseteq\mathcal{V}$,  there is some $k\in\omega$ and $V\in\mathcal{V}$ such that $j>k\Rightarrow V_j\subseteq V$.
\end{example}

This sensitivity of $\check{\mathrm{H}}^0(\,\cdot\,,\mathcal{D}_A)$ to the cardinality $\aleph_0$ is the beginning of a much more general phenomenon.

Example \ref{zerogroup} also relates to the previous example via the observation that, for any ordinal $\delta$, $$\check{\mathrm{H}}^0(\delta,\mathcal{D}_A)=\mathtt{sh}(\mathcal{D}_A)(\delta),$$ as the reader may verify. The groups $\check{\mathrm{H}}^0(\,\cdot\,,\,\cdot\,)$ in fact afford us a more concrete characterization of the sheafification process:

\begin{definitionandproposition}[Example 6.7.13 of \cite{spanier}] \label{defprop2} The functor $U\mapsto\check{\mathrm{H}}^0(U,\mathcal{P})$, together with the natural restriction maps, defines the sheafification $\mathtt{sh}(\mathcal{P})$ of a presheaf $\mathcal{P}$ on a topological space $X$.
\end{definitionandproposition}

For the remainder of this section, let $\mathcal{U}_\delta$ denote $\{\alpha\,|\,\alpha\in\delta\}$, the initial segments open cover of a limit ordinal $\delta$. More generally, write $\mathcal{U}_C$ for any cofinal $C\subseteq\delta$ viewed as a cover.

\begin{lemma}\label{apply}
Let $\delta$ be an ordinal of uncountable cofinality. A nontrivial coherent \linebreak $\Phi=\{\varphi_\alpha:\alpha\rightarrow A\,|\,\alpha\in\delta\}$ witnesses that $\mathrm{H}^1(\mathcal{U}_{\delta},\mathcal{D}_A)\neq 0$, in the following sense:

For $\alpha<\beta<\delta$, let $f^{\Phi}:\,(\alpha,\beta)\,\mapsto\,\varphi_\beta\!\!\upharpoonright\!\alpha - \varphi_\alpha$. Then $0\neq [f^{\Phi}]\in\mathrm{H}^1(\mathcal{U}_\delta,\mathcal{D}_A)$.

\end{lemma}

\begin{proof}
Coherence ensures that $f^\Phi\in L^1(\mathcal{U}_\delta,\mathcal{D}_A)$. 
As $f^\Phi$ is a coboundary, it is in $\text{ker}\,d^1$: for all $\alpha<\beta<\gamma<\delta$,
\begin{align}\label{1pt5} d^1 \! f^\Phi(\alpha,\beta,\gamma) & = f^\Phi(\beta,\gamma)-f^\Phi(\alpha,\gamma)+f^\Phi(\alpha,\beta) \\ & = (\varphi_\gamma-\varphi_\beta-\varphi_\gamma+\varphi_\alpha+\varphi_\beta-\varphi_\alpha)\!\restriction\!\alpha =0.\nonumber\end{align}
If, however, $f^\Phi=d^{0\!}(g)$ for some $g\in L^0(\mathcal{U}_\delta,\mathcal{D}_A)$, so that \begin{align}\label{geez}\varphi_\beta-\varphi_\alpha=g(\beta)\!\restriction\!\alpha\!-\,g(\alpha)\textnormal{ for all }\alpha<\beta<\delta,\end{align} then \begin{align}\label{LIMLIM}\varphi:=\varinjlim_{\beta<\delta} (\varphi_\beta-g(\beta))\end{align}
would trivialize the nontrivial family $\Phi$, a contradiction. Hence $$0\neq [f^{\Phi}]\in\mathrm{H}^1(\mathcal{U}_\delta,\mathcal{D}_A).$$\end{proof}

By the above lemma together with the following argument, for every $\delta$ of cofinality $\omega_1$ the function $\rho_1$ induces a nontrivial cohomology group $\mathrm{H}^1(\mathcal{U}_\delta,\mathcal{D}_\mathbb{Z})$. A third rho function, the \textit{last step function}  $\rho_3:[\omega_1]^2\rightarrow\{0,1\}$, effects a stronger corollary:

\begin{corollary}\label{DA} Let $\delta$ be an ordinal of cofinality $\omega_1$. Then $\mathrm{H}^1(\mathcal{U}_\delta,\mathcal{D}_A)\neq 0$ for any nontrivial abelian group $A$.
\end{corollary}

\begin{proof} Under a mild (ZFC) strengthening of the assumption $(\star)$, the family $$\{\rho_3(\,\cdot\,,\beta)\,|\,\beta\in\omega_1\}$$ is nontrivial coherent in the sense of Definition \ref{fixntc} (see \cite{walks} II.4). We make this assumption. Strictly speaking, $\rho_3$ should be group-valued for Definition \ref{fixntc} to apply, but of course this is the point: for \textit{any} fixed abelian group $A$ and nonzero $a\in A$, the map $e: 0\mapsto 0,\,1\mapsto a$ determines a nontrivial coherent $A$-valued family of functions $\{e\circ\rho_3(\,\cdot\,,\beta)\,|\,\beta\in\omega_1\}$.

Now fix $\delta\in\text{Cof}\,(\omega_1)$ and a closed and cofinal $C_\delta=\{\eta_i^\delta\,|\,i<\omega_1\}\subseteq\delta$. For all $\alpha<\beta<\delta$, let
\smallskip
\begin{equation}\label{stretchedrho3}
  \varphi^{\delta}_\beta(\alpha) =
  \begin{cases}
                                   e\circ\rho_3(\text{otp}(C_\delta\cap\alpha),\text{otp}(C_\delta\cap\beta)) & \text{if $\alpha\in C_\delta$} \\
                                   0 & \text{otherwise.}
  \end{cases}
\end{equation}
\smallskip
Then $\Phi=\{\varphi^{\delta}_\beta:\beta\to A\,|\,\beta\in\delta\}$ is a nontrivial coherent family of functions on $\delta$. Lemma \ref{apply} completes the proof.
\end{proof}
A perfectly analogous argument applies to $\mathcal{A}_d$ and $\rho_2$, for $A=\mathbb{Z}$: by Corollary \ref{locallyconstantcoherence}, $\rho_2$ witnesses that $\mathrm{H}^1(\mathcal{U}_{\omega_1},\mathbb{Z}_d)\neq 0$, again via a function
$$f:\,(\alpha,\beta)\,\mapsto\,\rho_2(\,\cdot\,,\beta)\!\restriction\!\alpha-\,\rho_2(\,\cdot\,,\alpha),$$
just as above. Again a closed and cofinal $C_\delta\subseteq\delta$ ``stretches'' this witness to one for $\mathrm{H}^1(\mathcal{U}_\delta,\mathbb{Z}_d)\neq 0$, via the function
\begin{align}\label{stretchedrho2}\rho_2^\delta(\alpha,\beta):=\rho_2(\text{otp}(C_\delta\cap\alpha),\text{otp}(C_\delta\cap\beta))\hspace{.3 cm}\textnormal{ for }\alpha<\beta<\delta\end{align}
The similarities of these constructions reflect a more general principle --- namely, the essential 
equivalence, in our contexts, of $\mathcal{D}_A$ (encoding the set-theoretic theme of \textit{mod finite} relations) and $\mathcal{A}_d$ (standard in more geometric settings):
\begin{lemma}\label{gddg} For cofinal $C\subseteq\delta\cap\text{Lim}$, the cochain complexes $\mathcal{L}(\mathcal{U}_C,\mathcal{A}_d)$ and \linebreak $\mathcal{L}(\mathcal{U}_C,\mathtt{sh}(\mathcal{D}_A))$ are chain-isomorphic, and the complexes $\mathcal{L}(\mathcal{U}_C,\mathtt{sh}(\mathcal{D}_A))$ and \linebreak $\mathcal{L}(\mathcal{U}_C,\mathcal{D}_A)$ are chain homotopy equivalent. In consequence, $\mathrm{H}^n(\mathcal{U}_C,\mathcal{D}_A)\cong\mathrm{H}^n(\mathcal{U}_C,\mathcal{A}_d)$ for all $n\geq0$.
\end{lemma}
\begin{proof} Fix $\gamma\in C$ and define the following maps: for $\varphi\in \mathcal{A}_d(\gamma)$ and $\beta\in\gamma$, let $e_{\varphi}(\beta)$ denote the eventual value of $\varphi$ below $\beta$ (in particular, $e_{\varphi}(\alpha+1)=\varphi(\alpha)$). As $\varphi$ is locally constant, this is well-defined --- except at 0: let $e_{\varphi}(0)=0$. Define then $\partial\varphi\in \mathtt{sh}(\mathcal{D}_A)(\gamma)$ by
$$\partial\varphi(\beta)=\varphi(\beta)-e_{\varphi}(\beta)$$for $\beta\in\gamma$. This operation is reversible: for $\psi\in\mathtt{sh}(\mathcal{D}_A)(\gamma)$, define $\partial^{-1}\psi$ by $$\partial^{-1}\psi(\beta)=\sum_{\alpha\leq\beta}\psi(\alpha)$$for $\beta\in\gamma$. Clearly $\partial^{-1}\circ\partial=\text{id}\,:\,\mathcal{A}_d(\gamma)\rightarrow\mathcal{A}_d(\gamma)$. $\partial^{-1}$ isn't quite an isomorphism from $\mathtt{sh}(\mathcal{D}_A)$ to $\mathcal{A}_d$, though: $\partial^{-1}\psi\notin\mathcal{A}_d(\gamma)$ if $\text{supp}(\psi)\cap\text{Lim}\neq\varnothing$. However, this is easily addressed: for $\psi\in\mathtt{sh}(\mathcal{D}_A)(\gamma)$, define
\smallskip
\begin{equation}\label{bumpup}
  r\psi(\beta) =
  \begin{cases}
                                  \psi(\beta) & \text{if $\beta\in\omega$} \\
                                   0 & \text{if $\beta$ is a limit} \\
                                   \psi(\alpha) & \text{if $\beta=\alpha+1$ and $\beta>\omega$}
  \end{cases}
\end{equation}
\smallskip
  for $\beta\in\gamma$. For $\varphi\in\mathcal{A}_d(\gamma)$, define
\smallskip
\begin{equation}\label{bumpdown}
  r^{-1}\varphi(\beta) =
  \begin{cases}
                                  \varphi(\beta) & \text{if $\beta\in\omega$} \\
                                   \varphi(\beta+1) & \text{otherwise}
  \end{cases}
\end{equation}
\smallskip
  for $\beta\in\gamma$. Then $\partial^{-1}\circ r\,:\,\mathtt{sh}(\mathcal{D}_A)(\gamma)\rightarrow\mathcal{A}_d(\gamma)$ and $r^{-1}\circ \partial\,:\,\mathcal{A}_d(\gamma)\rightarrow\mathtt{sh}(\mathcal{D}_A)(\gamma)$ are inverse, as desired. They also extend to cochain isomorphisms, as the reader may verify: both
\begin{align*}
\mathcal{L}(\mathcal{U}_C,\mathcal{A}_d)\xrightarrow{r^{-1}\partial}\mathcal{L}(\mathcal{U}_C,\mathtt{sh}(\mathcal{D}_A))\xrightarrow{\partial^{-1}r}\mathcal{L}(\mathcal{U}_C,\mathcal{A}_d)\end{align*}
and
\begin{align*}
\mathcal{L}(\mathcal{U}_C,\mathtt{sh}(\mathcal{D}_A))\xrightarrow{\partial^{-1}r}\mathcal{L}(\mathcal{U}_C,\mathcal{A}_d)\xrightarrow{r^{-1}\partial}\mathcal{L}(\mathcal{U}_C,\mathtt{sh}(\mathcal{D}_A))\end{align*}
are the identity.

We show now that the cochain complexes $\mathcal{L}(\mathcal{U}_C,\mathtt{sh}(\mathcal{D}_A))$ and $\mathcal{L}(\mathcal{U}_C,\mathcal{D}_A)$ are chain homotopy equivalent. To this end, we define chain maps $i:\mathcal{L}(\mathcal{U}_C,\mathcal{D}_A)\to\mathcal{L}(\mathcal{U}_C,\mathtt{sh}(\mathcal{D}_A))$  and $m:\mathcal{L}(\mathcal{U}_C,\mathtt{sh}(\mathcal{D}_A))\to\mathcal{L}(\mathcal{U}_C,\mathcal{D}_A)$ and chain homotopies \linebreak$s:\mathcal{L}(\mathcal{U}_C,\mathtt{sh}(\mathcal{D}_A))\to\mathcal{L}(\mathcal{U}_C,\mathtt{sh}(\mathcal{D}_A))$ and $t:\mathcal{L}(\mathcal{U}_C,\mathcal{D}_A)\to\mathcal{L}(\mathcal{U}_C,\mathcal{D}_A)$ such that \begin{align}\label{eleven}
im-\mathrm{id} & =ds+sd \\ \label{twelve} mi-\mathrm{id} & =dt+td
\end{align}
both hold.

The map $i$ is simply inclusion: for $f\in L^k(\mathcal{U}_C,\mathcal{D}_A)$ and $\vec{\alpha}\in [C]^{k+1}$ let $i(f)(\vec{\alpha})=f(\vec{\alpha})$. To define $m$, first fix an $\bar{m}:C\to C$ with $\alpha<\bar{m}(\alpha)<\bar{m}(\beta)$ for all $\alpha<\beta$ in $C$. Denote also by $\bar{m}$ this map's natural extension to maps of ordered tuples $(\alpha_0,\dots,\alpha_j,\dots)\mapsto(\bar{m}(\alpha_0),\dots,\bar{m}(\alpha_j),\dots)$. For $f\in L^k(\mathcal{U}_C,\mathtt{sh}(\mathcal{D}_A))$ and $\vec{\alpha}\in [C]^{k+1}$ let $m(f)(\vec{\alpha})=f(\bar{m}(\vec{\alpha}))\!\restriction\!\alpha_0$. We omit restriction-notations for the remainder of this proof. 

The sequence $s=\langle\,s^k: L^k(\mathcal{U}_C,\mathtt{sh}(\mathcal{D}_A))\to L^{k-1}(\mathcal{U}_C,\mathtt{sh}(\mathcal{D}_A))\,|\,k\in\omega\,\rangle$ is defined as follows: for $\vec{\alpha} \in [C]^k$ and $k\geq 1$ and $f\in L^k(\mathcal{U}_C,\mathtt{sh}(\mathcal{D}_A))$,
$$s^k(f)(\vec{\alpha})=\displaystyle\sum_{j=0}^{k-1} \,(-1)^j f(\alpha_0,\dots,\alpha_j,\bar{m}(\alpha_j),\dots,\bar{m}(\alpha_k))$$
Necessarily, $s^0$ the zero map; hence the relation
$$f(\bar{m}(\alpha))-f(\alpha)=d^0(f)(\alpha,\bar{m}(\alpha))=0+s^1d^0(f)(\alpha)$$
witnesses the degree-zero instance of equation \ref{eleven}. Equation \ref{eleven} follows more generally from the relation below holding for any $f\in L^k(\mathcal{U}_C,\mathtt{sh}(\mathcal{D}_A))$ with $k\geq 1$: $$f(\bar{m}(\vec{\alpha}))-f(\vec{\alpha})=d^{k-1}s^k(f)(\vec{\alpha})+s^{k+1}d^k(f)(\vec{\alpha})$$
As the reader may verify, the essential point above is that  the terms on the right-hand side with ``hybrid'' arguments $(\alpha_0,\dots,\alpha_j,\bar{m}(\alpha_j),\dots,\bar{m}(\alpha_k))$ all cancel.\footnote{This equation and that for the chain-homotopy $s$ bear comparison with the prism operators standard, for example, in arguments for the homotopy invariance of singular homology (cf. \cite[p. 112]{hatcher}).} Clearly an identical construction and argument will establish equation \ref{twelve} as well; this concludes the proof of the theorem.
\end{proof}

\begin{corollary} \label{omega_one_cor}
Let $\delta$ be an ordinal of cofinality $\omega_1$. Then $\mathrm{H}^1(\mathcal{U}_\delta,\mathcal{A}_d)\neq 0$ for any nontrivial abelian group $A$.
\end{corollary}

Operative in the above arguments was the following:
\begin{observation}\label{transformation} Nontrivial coherence exhibits degrees of transformation invariance befitting a homological property. In particular, we will increasingly view the following principles as routine:
\begin{enumerate}
\item[(i)] Shifts $($as in \textnormal{(\ref{bumpup}), (\ref{bumpdown})}$))$ and re-scalings $($as in \textnormal{(\ref{stretchedrho3}), (\ref{stretchedrho2}}$))$ of a nontrivial coherent family are again nontrivial coherent. In particular, cofinal $C_\delta\subseteq\delta$ of order-type $\text{cf}(\delta)=\kappa$ ``stretch'' or propagate nontrivial coherence from $\kappa$ to $\delta$. This holds for higher-order coherence as well; for this reason, in Section \ref{sectiontwo}, we almost exclusively consider regular cardinals $\kappa$, leaving arguments' extensions to the class of ordinals $\text{Cof}(\kappa)$ to the reader.
\item[(ii)] Fix $\eta\in\delta$ and cofinal $C\subseteq(\eta,\delta)$. A rough converse to (1) is that \textit{the restriction, $\Phi[\eta,C]:=\{\varphi_\beta\!\!\upharpoonright\![\eta,\beta)|\,\beta\in C\}$, of a nontrivial coherent $\Phi=\{\varphi_\beta\,|\,\beta\in\delta\}$ to cofinal indices and/or to cobounded domain is again nontrivial coherent.} 
\end{enumerate}
\end{observation}
\noindent Both properties hold much more generally. They hold, in particular, for all non-$n$-trivial $n$-coherence relations; these we now define.

\subsection{Higher dimensional coherence}\label{I.3}\hfill

Henceforth for readability we leave many of the more obvious restrictions unexpressed, writing $\varphi_{\vec{\alpha}^0}$ for $\varphi_{\vec{\alpha}^0}\!\restriction\!\alpha_0$, for example, as in the sum (\ref{ncoh}) below. The nontrivial coherence of Definition \ref{fixntc} is the case $n=1$ of the following:

\begin{definition}\label{highernontriv}
Let $n$ be a positive integer and let $\Phi_n=\{\varphi_{\vec{\alpha}}:\alpha_0\rightarrow A\,|\,\vec{\alpha}\in [\delta]^{n}\}$ 
be a family of functions. (We say in this case that $\Phi_n$ is \emph{$A$-valued} and \emph{of height} $\delta$.) 
\begin{itemize}
	\item $\Phi_n$ is \emph{$n$-coherent} if 
		\begin{align}
			\label{ncoh}\sum_{i=0}^{n} (\text{-}1)^i\varphi_{\vec{\alpha}^i}=^{*} 0
		\end{align}
		for all $\vec{\alpha}\in [\delta]^{n+1}$.
	\item $\Phi_n$ is \emph{$n$-trivial} if:
		\begin{itemize}
			\item $n = 1$ and $\Phi_n$ is trivial, or
			\item $n > 1$ and there exists $\Psi_{n-1}=\{\psi_{\vec{\alpha}}:\alpha_0\rightarrow A\,|\,\vec{\alpha}\in [\delta]^{n-1}\}$ 
				such that
				\begin{align}\label{ntriv}
					\sum_{i=0}^{n-1} (\text{-}1)^i\psi_{\vec{\alpha}^i}=^{*} \varphi_{\vec{\alpha}}
				\end{align}
				for all $\vec{\alpha}\in [\delta]^{n}$.
		\end{itemize}
\end{itemize}
\end{definition} 

Observe that non-$n$-trivial $n$-coherence so defined exhibits the signature features of nontrivial coherence. Namely:
\begin{enumerate}
\item[(i)] \emph{Rigidity of dimension}. Non-$n$-trivial $n$-coherence names a (mod finite) redundancy of $(n+1)$-dimensional data (\ref{ncoh}) that defies approximation (\ref{ntriv}), nevertheless, on any smaller number of dimensions.
\item[(ii)] \emph{Incompactness}. Just as for the case $n=1$, $n$-coherence and $n$-triviality are two faces of the same thing: a family is $n$-coherent if and only if its proper initial segments are each $n$-trivial. Equivalently, a family is $n$-trivial if and only if some $n$-coherent family properly extends it.
In either view, non-$n$-trivial $n$-coherence on $\delta$ marks an abrupt \textit{change in behavior at}  $\delta$ --- a phenomenon known in set-theoretic circles as \emph{incompactness}, as noted above \cite{compactness}. This may be clearest in an example.
\end{enumerate}
\begin{example}\label{elpmaxe}
$\Phi_2=\{\varphi_{\alpha\beta}:\alpha\rightarrow A\,|\,(\alpha,\beta)\in [\delta]^2\}$ is 2-coherent if
\begin{align}\label{113}\varphi_{\beta\gamma}\!\restriction\!\alpha-\,\varphi_{\alpha\gamma}+\varphi_{\alpha\beta}=^{*}0\hspace{.45 cm}\textnormal{ for all }(\alpha,\beta,\gamma)\in [\delta]^3\end{align}
$\Phi_2$ is 2-trivial if there exists a $\Psi_1=\{\psi_\alpha\,|\,\alpha\in\delta\}$ such that
\begin{align*}\psi_\beta\!\restriction\!\alpha-\,\psi_\alpha=^{*}\varphi_{\alpha\beta}\hspace{.45 cm}\textnormal{ for all }(\alpha,\beta)\in [\delta]^2\end{align*} or equivalently if \begin{align}\label{114}-\psi_\beta\!\restriction\!\alpha-\,-\psi_\alpha+\varphi_{\alpha\beta}=^{*}0\end{align}for all $(\alpha,\beta)\in [\delta]^2$. By (\ref{113}), the family $\{\psi_\beta:=-\varphi_{\beta\gamma}\,|\,\beta\in\gamma\}$ fulfills (\ref{114}), for all $(\alpha,\beta)\in [\gamma]^2$. In other words, this family 2-trivializes the initial segment $\{\varphi_{\alpha\beta}\,|\,(\alpha,\beta)\in [\gamma]^2\}$ of $\Phi_2$.

This holds very generally: fix $n>1$ and let $\Phi_n=\{\varphi_{\vec{\alpha}}:\alpha_0\rightarrow A\,|\,\vec{\alpha}\in [\delta]^n\}$ be $n$-coherent, and let $\Phi_n^\gamma:=\{\varphi_{\vec{\alpha}\gamma}\,|\,\vec{\alpha}\in[\gamma]^{n-1}\}$ and $\Phi_n\!\restriction\!\gamma\,:=\{\varphi_{\vec{\alpha}}\,|\,\vec{\alpha}\in [\gamma]^n\}$, for any $\gamma<\delta$. The fact that \begin{align}(-1)^{n+1}\Phi_n^\gamma\,\textnormal{ $n$-trivializes }\,\Phi_n\!\!\upharpoonright\!{\gamma}\end{align} is then immediate from Definition \ref{highernontriv}.
Relatedly, and almost as immediately,
\begin{align}\label{exercise}\Phi_n^\delta\!\restriction\!\gamma -\,\Phi_n^\gamma\,:=\{\varphi_{\vec{\alpha}\delta}-\varphi_{\vec{\alpha}\gamma}\,|\,\vec{\alpha}\in[\gamma]^{n-1}\}\hspace{.081 cm}\textnormal{ is ($n-1$)-trivial}.\end{align}Of course,\begin{align}\label{cook}\Phi_n=\bigcup_{\gamma\in\delta}\Phi_n^\gamma\end{align}and this fact, together with (\ref{exercise}), shapes the following view:
\end{example}
\begin{enumerate}
\item[(iii)] \emph{An $n$-coherent $\Phi_n$ is a family of families of functions, agreeing modulo  $(n-1)$-} \linebreak \emph{triviality.}

This relates to the case $n=1$ in several ways:
\begin{enumerate}
\item In Corollary \ref{aronszajn}, the initial segments of a 1-coherent $$\Phi_1=\{\varphi_\gamma\,|\,\gamma\in\delta\}=\bigcup_{\gamma\in\delta}\{\varphi_\gamma\}$$ form a height-$\delta$ tree of functions $T(\Phi_1)$ possessing a cofinal branch if and ony if $\Phi_1$ is 1-trivial. Decomposition (\ref{cook}) determines a tree of \textit{families of} functions $T(\Phi_n)$, again of height $\delta$ and possessing a cofinal branch only if $\Phi_n$ is $n$-trivial.
\item This view suggests an equation of \emph{0-trivial} with \emph{finitely supported.} The degree zero nontrivial coherent objects are then, on this suggestion, the very ``material'' of non-1-trivial 1-coherence: they are cofinal sets $C_\beta\subseteq\beta$ of order-type $\text{cf}(\beta)=\omega$, the fundamental inputs in the arguments of Section \ref{I.1}. Such sets are sometimes termed \emph{ladders}.
\end{enumerate}
\end{enumerate}
\begin{definition}\label{zerotriv} A function $\varphi:\beta\rightarrow A$ is \emph{0-trivial} if $$\varphi=^{*} 0$$ i.e., if $\varphi$ has finite support. Just as for higher $n$, then, \emph{0-coherent} means \emph{0-trivial on all initial segments}: $\varphi$ is \textit{0-coherent} if
$$\varphi\!\restriction\!\alpha\,=^{*} 0\hspace{.334 cm}\textnormal{ for all }\alpha<\beta.$$
\end{definition}
Hence a non-0-trivial 0-coherent function is simply a $\varphi:\beta\rightarrow A$ supported on some ladder $C_\beta\subseteq\beta\in\text{Cof}(\omega)$. Under this definition, non-0-trivial 0-coherence displays as plain an affinity for $\omega$ as non-1-trivial 1-coherence does for $\omega_1$: for both $n=0$ and $n=1$, the least ordinal and, consistently, the only cofinality,\footnote{For the $n=1$ case, see Theorems \ref{goblot}, \ref{stronglycmpct}, and \ref{pidtheorem}, below.} admitting non-$n$-trivial $n$-coherent functions, is $\omega_n$.

We note lastly that, as mentioned, higher non-$n$-trivial $n$-coherence has all the ``resilience'' of classical nontrivial coherence:
\begin{enumerate}
\item[(iv)] \emph{Transformation invariance}. The restriction $\Phi[\eta,C]$ (as in Observation \ref{transformation}) of a non-$n$-trivial $n$-coherent family $\Phi$ is again non-$n$-trivial $n$-coherent. This is largely because, as noted, $n$-coherent families are $n$-trivial on initial segments, so that global $n$-triviality is a question only really of the tail.
\end{enumerate}

\subsection{Higher dimensional cohomology}\label{I.4}\hfill

Higher coherence and triviality, of course, are above all cohomological phenomena: for non-$n$-trivial $n$-coherent $\Phi=\{\varphi_{\vec{\alpha}}:\alpha_0\rightarrow A\,|\,\vec{\alpha}\in [\delta]^{n}\}$ define, as before, elements $f^\Phi$ of $L^n(\mathcal{U}_\delta,\mathcal{D}_A)$ by
$$f^\Phi:\,\vec{\alpha}\mapsto \sum_{i=0}^{n} (\text{-}1)^i\varphi_{\vec{\alpha}^i}$$
Again, then,
$$0\neq[f^\Phi]\in \mathrm{H}^n(\mathcal{U}_\delta,\mathcal{D}_A)$$
as the reader may verify. In fact something stronger is true: that \textit{all} the elements of $\mathrm{H}^n(\mathcal{U}_\delta,\mathcal{D}_A)$ are of this form, and that the map $[\Phi]\mapsto [f^\Phi]$ defines an isomorphism between ``mod finite'' cohomology groups and ``finite support'' cohomology groups of the next degree (see (\ref{readone}) and (\ref{readtwo}), below). This, in turn, is an instance of the following.

\begin{definition} Let $\mathcal{P}_A$ denote a presheaf of functions to $A$ (recall that $\mathcal{D}_A$, $\mathcal{A}_d$, and $\mathcal{A}$ are examples).\footnote{More generally, any sheaf is a presheaf of functions to the union $A$ of its stalks \cite{bredonsheaves}.} Write $\mathcal{E}_A$ for the presheaf $U\mapsto \prod_U A$, and $\mathcal{F}_A$ for $\mathcal{E}_A/\mathcal{P}_A:U\mapsto(\mathcal{E}_A(U)/\mathcal{P}_A(U))$.
\end{definition}

Throughout this section, view the presheaves in question as all defined on some fixed $\delta$ of uncountable cofinality.

The natural inclusion and quotient maps,  $\mathcal{P}_A\rightarrow\mathcal{E}_A$ and $\mathcal{E}_A\rightarrow\mathcal{F}_A$, respectively, induce a short exact sequence of cochain complexes ---
\begin{align}
\mathcal{L}(\mathcal{V},\mathcal{P}_A)\rightarrow \mathcal{L}(\mathcal{V},\mathcal{E}_A)\rightarrow\mathcal{L}(\mathcal{V},\mathcal{F}_A)
\end{align}
--- and, in consequence, induce the long exact sequence
\begin{align}
&\mathrm{H}^0(\mathcal{V},\mathcal{P}_A)\hookrightarrow \mathrm{H}^0(\mathcal{V},\mathcal{E}_A)\rightarrow\mathrm{H}^0(\mathcal{V},\mathcal{F}_A)\xrightarrow{d}\mathrm{H}^1(\mathcal{V},\mathcal{P}_A)\rightarrow\mathrm{H}^1(\mathcal{V},\mathcal{E}_A)\rightarrow\dots\label{aitchone}\\ &\dots\rightarrow\mathrm{H}^{n}(\mathcal{V},\mathcal{E}_A)\rightarrow\mathrm{H}^n(\mathcal{V},\mathcal{F}_A)\xrightarrow{d} \mathrm{H}^{n+1}(\mathcal{V},\mathcal{P}_A)\rightarrow\mathrm{H}^{n+1}(\mathcal{V},\mathcal{E}_A)\rightarrow\dots \label{aitchtwo}
\end{align}
We leave to the reader the straightforward verification that $\mathrm{H}^n(\,\cdot\, ,\mathcal{E}_A)=0$ for $n>0$; one might more abstractly deduce this from the fact that $\mathcal{E}_A$ is \textit{flasque}, i.e., flabby \cite{bredonsheaves}. By (\ref{aitchone}), then, 
\begin{align}\label{readone}
\frac{\mathrm{H}^0(\mathcal{V},\mathcal{F}_A)}{\text{im}(\mathrm{H}^0(\mathcal{V},\mathcal{E}_A))}\stackrel{d}{\cong}\mathrm{H}^1(\mathcal{V},\mathcal{P}_A)
\end{align}
Specialization to the case of $\mathcal{V}=\mathcal{U}_\delta$ and $\mathcal{P}_A=\mathcal{D}_A$ reads \begin{align*}\mathrm{H}^1(\mathcal{U}_\delta,\mathcal{D}_A) \textit{ is the group of coherent families of $A$-valued functions indexed by $\delta$,} \\ \textit{quotiented by the group of trivial families of $A$-valued functions indexed by $\delta$.}\end{align*} The group operation here is the natural one: for coherent $\Phi=\{\varphi_\alpha\,|\,\alpha\in\delta\}$ and $\Psi=\{\psi_\alpha\,|\,\alpha\in\delta\}$, the sum $\Phi+\Psi=\{\varphi_\alpha+\psi_\alpha\,|\,\alpha\in\delta\}$ is again coherent. The map $d$ is the familiar $[\Phi]\mapsto [f^{\Phi}]$.

Similarly, (\ref{aitchtwo}) specializes to
\begin{align}\label{readtwo}\mathrm{H}^n(\mathcal{U}_\delta,\mathcal{F}_A)\cong\mathrm{H}^{n+1}(\mathcal{U}_\delta,\mathcal{D}_A)\hspace{.5 cm}\textnormal{ for }n>1,
\end{align}
again via the mapping $[\Phi]\mapsto [f^{\Phi}]$. This relation, of course, determined Definition \ref{highernontriv}, through which
the above reading generalizes:

\begin{theorem}\label{groupreading} For $n>0$, $\mathrm{H}^n(\mathcal{U}_\delta,\mathcal{D}_A)$ is the group of $n$-coherent families of $A$-valued functions indexed by $\delta$, quotiented by the group of $n$-trivial families of $A$-valued functions indexed by $\delta$.\end{theorem}
Addition in these groups is ``pointwise,'' as in the case $n=1$.

The quotient groups of Theorem \ref{groupreading} are in fact the full \v{C}ech cohomology groups of a given $\delta$:
\begin{theorem}\label{combicasting} $\check{\mathrm{H}}^n(\delta,\mathcal{P}_A)\cong \mathrm{H}^n(\mathcal{V},\mathcal{P}_A)$, for any $n>0$ and $\mathcal{V}\geq\mathcal{U}_\delta$ and presheaf $\mathcal{P}_A$ of functions to $A$. In particular, $\check{\mathrm{H}}^n(\delta,\mathcal{A}_d)\cong\check{\mathrm{H}}^n(\delta,\mathcal{D}_A)\cong\mathrm{H}^n(\mathcal{U}_\delta,\mathcal{D}_A)$ for all $n>0:$ each is naturally isomorphic to the group described in Theorem \ref{groupreading}.
\end{theorem}

\begin{proof} The second assertion follows from the first, together with Lemma \ref{gddg}: $\mathrm{H}^n(\mathcal{U}_\delta,\mathcal{D}_A)\cong\check{\mathrm{H}}^n(\delta,\mathcal{D}_A)\cong\mathrm{H}^n(\mathcal{U}_C,\mathcal{D}_A)\cong\mathrm{H}^n(\mathcal{U}_C,\mathcal{A}_d)\cong\check{\mathrm{H}}^n(\delta,\mathcal{A}_d)$, where $C=\delta\cap\text{Lim}$. We turn then to the first assertion. The heart of the argument is the following square:

\begin{align}\label{chartone}
\xymatrix{\mathrm{H}^n(\mathcal{U}_\delta,\mathcal{F}_A) \ar[r]^d \ar[d]_{r^{*}_{\mathcal{V}\mathcal{U}_\delta}} & \mathrm{H}^{n+1}(\mathcal{U}_\delta,\mathcal{P}_A) \ar[d]^{r^{*}_{\mathcal{V}\mathcal{U}_\delta}}  \\ \mathrm{H}^n(\mathcal{V},\mathcal{F}_A) \ar[r]^d & \mathrm{H}^{n+1}(\mathcal{V},\mathcal{P}_A) }
\end{align}

In the case of $n=0$, replace the left-hand side with the quotients of (\ref{readone}): the point in all cases is that\begin{enumerate}
\item The lateral maps, $d$, are isomorphisms, and
\item The diagram commutes.
\end{enumerate}
Therefore it suffices to show $r^{*}_{\mathcal{V}\mathcal{U}_\delta}:\,\mathrm{H}^n(\mathcal{U}_\delta,\mathcal{F}_A)\rightarrow\mathrm{H}^n(\mathcal{V},\mathcal{F}_A)$ to be an isomorphism; $\check{\mathrm{H}}^{k+1}(\delta,\mathcal{P}_A)$ then will be a direct limit of isomorphisms, and the theorem will follow. This in turn follows from two observations:
\begin{enumerate}
\item By Corollary \ref{thesubcover}, $\mathcal{V}$ contains a collection of the form $\{\,[\eta,\xi)\,|\,\xi\in C\}$, for some $C$ cofinal in $\delta$, in the sense that there exists some $\{V_\xi\,|\,\xi\in C\}\subseteq\mathcal{V}$ with $[\eta,\xi)\subseteq V_\xi$ for all $\xi\in C$.
\item Fix representative $\Phi\in L^k(\mathcal{U}_\delta,\mathcal{E}_A)$ for a class $[\Phi]\in \mathrm{H}^k(\mathcal{U}_\delta,\mathcal{F}_A)$. In this generalized context, the logic of ``if $\Phi$ is non-$n$-trivial $n$-coherent, then so too is $\Phi[\eta, C]$'' continues to apply; here it takes the form
$$0\neq [\Phi]\in\mathrm{H}^k(\mathcal{U}_\delta,\mathcal{F}_G)\textnormal{ implies that }0\neq [r_{\mathcal{V}\mathcal{U}_\delta}(\Phi)]\in\mathrm{H}^k(\mathcal{V},\mathcal{F}_G)$$
This shows that $r^{*}_{\mathcal{V}\mathcal{U}_\delta}:\,\mathrm{H}^k(\mathcal{U}_\delta,\mathcal{F}_A)\rightarrow\mathrm{H}^k(\mathcal{V},\mathcal{F}_A)$ is injective. The reverse observation --- that any non-$n$-trivial $n$-coherent $\Phi[\eta, C]$ extends to a classically non-$n$-trivial $n$-coherent $\Phi$ --- shows that $r^{*}_{\mathcal{V}\mathcal{U}_\delta}$ is surjective.\end{enumerate} Hence $r^{*}_{\mathcal{V}\mathcal{U}_\delta}:\,\mathrm{H}^k(\mathcal{U}_\delta,\mathcal{F}_A)\rightarrow\mathrm{H}^k(\mathcal{V},\mathcal{F}_A)$ is a isomorphism. In consequence, $r^{*}_{\mathcal{V}\mathcal{U}_\delta}:\,\mathrm{H}^{k+1}(\mathcal{U}_\delta,\mathcal{P}_A)\rightarrow\mathrm{H}^{k+1}(\mathcal{V},\mathcal{P}_A)$ is as well, as desired.
\end{proof}

\begin{remark} Essentially,  $\mathcal{U}_\delta$ functions above as a \textit{good cover}; in other words, it marks a stage at which the limit $\check{\mathrm{H}}^1(\delta,\mathcal{P})$ has already arrived. This is the ambiguous effect of the Pressing Down Lemma: at once it (1) rules out paracompactness (Corollary \ref{thesubcover}), which is the usual environment of good covers, and (2) it enforces the effects of good covers.
\end{remark} 

Theorems \ref{groupreading} and \ref{combicasting} together
describe a combinatorial translation of mixed effect: in the case of $n=1$, they cast that most basic of topological invariants --- cohomology over the constant sheaf --- as a strict measure of that centerpiece of infinitary combinatorics, \textit{nontrivial coherence}. It frames the higher cohomology groups, on the other hand, as phenomena largely without antecedent or meaning in set-theoretic experience: \textit{non-$n$-trivial $n$-coherence}. Why? One reason, as noted, is that non-$2$-trivial $2$-coherence, for example, is simply imperceptible below $\omega_2$. More generally, $\check{\mathrm{H}}^n(\delta,\mathcal{A}_d)=0$ for any $\delta<\omega_n$, a phenomenon we describe in the following section.

\subsection{ZFC constraints}\label{I.5}\hfill

We pause to remark how canonical the groups $\check{\mathrm{H}}^n(\,\cdot\,,\mathcal{A}_d)$ in fact are.
They are, for example, the Alexander-Spanier cohomology groups with coefficients in $A$ \cite{Dowker}; they manifest as well as Ext$^n$ and lim$^n$ groups, and it was in these guises, in homological and ring-theoretic settings, that their fundamental relations to the ordinals $\omega_n$ were first perceived. The surveys \cite{osofsky1}, \cite{osofsky2}, \cite{pierce}, \cite{jensencu}, \cite{rings}, \cite{husainov}, and \cite{strongshape} usefully summarize these recognitions, the most central of which date to the 1960s and early 1970s. Salient among these recognitions are vanishing and non-vanishing results appearing in \cite{goblot} and \cite{rings}, respectively, which manifest in our framework as Theorems \ref{goblot} and \ref{nonntrivncohomegan} below. These results together delimit the proper terrain of \emph{Section \ref{sectiontwo}: Consistency results}: the latter are only meaningfully available for groups $\check{\mathrm{H}}^n(\xi,\mathcal{A}_d)$ in which $\text{cf}(\xi)>\omega_n$.
\begin{theorem}\label{goblot} For any abelian group $A$ and positive integer $n$ and any $\delta$ of cofinality less than $\omega_n$, we have $\check{\mathrm{H}}^n(\delta,\mathcal{A}_d) = 0$.
\end{theorem}

\begin{proof} For the following argument (due essentially to \cite{goblot}), fix the abelian group $A$. By Theorem \ref{combicasting} it will suffice to prove the following:
\begin{claim}\label{31} Let $\text{cf}\,(\delta)=\omega_k$ and let $n$ be a positive integer greater than $k$. Any $n$-coherent $\Phi_n=\{\varphi_{\vec{\beta}}:\beta_0\rightarrow A\,|\,\vec{\beta}\in [\delta]^{n}\}$ is $n$-trivial.
\end{claim}

We argue the claim by induction on $k$. In the base case, $\delta$ is a successor; to streamline the argument, adopt the convention that $\text{cf}\,(\delta)$ then equals $\omega_{-1}$. Take any positive $n$ and $n$-coherent $\Phi$ and let $\gamma=\delta-1$. If $n=1$, then any $\varphi$ extending $\varphi_\gamma$ will trivialize $\Phi$. If $n>1$, then for any $\vec{\beta}\in [\gamma]^{n-1}$ and $\alpha<\beta_0$ let $\psi_{\vec{\beta}}(\alpha)=(-1)^{n-1}\varphi_{\vec{\beta}\gamma}(\alpha)$. For $\vec{\beta}\in [\delta]^{n-1}\backslash[\gamma]^{n-1}$ let $\psi_{\vec{\beta}}$ constantly equal zero. Then $\Psi=\{\psi_{\vec{\beta}}\,|\,\vec{\beta}\in [\delta]^{n-1}\}$ is an $n$-trivialization of $\Phi$, as the reader may verify. This completes the base case.

Suppose now that the claim holds for all $k<m$; we show that it holds for $k=m$ as well. To that end, fix $\delta\in\text{Cof}\,(\omega_m)$ and a closed unbounded $C_\delta=\{\xi_i\,|\,i<\omega_m\}\subseteq\delta$ and an $n$-coherent $\Phi$, with $n>m$. If $m=0$ and $n=1$ then $\varphi:\alpha\mapsto\varphi_{\xi(\alpha)}(\alpha)$ will $n$-trivialize $\Phi$, where $\xi(\alpha)=\min (C_\delta\backslash(\alpha+1))$. For any other $0\leq m<n$ we show how the inductive hypothesis enables a recursive construction of $n$-trivializations $\Psi_\ell$ of $\Phi\!\restriction\!\xi_\ell$ so that $\Psi:=\bigcup_{\ell<\omega_m}\!\Psi_\ell\,$ $n$-trivializes $\Phi$. This will conclude the induction step.

Begin the construction with any $n$-trivialization $\Psi_0$ of $\Phi\!\restriction\!\xi_0$. Suppose then that $n$-trivializations $\Psi_0\subset\dots\subset\Psi_j\subset\dots$ of $\Phi\!\restriction\!\xi_0\subset\dots\subset\Phi\!\restriction\!\xi_j\subset\dots$, respectively, have been constructed for all $j<\ell<\omega_m$. If $\ell$ is a limit ordinal, then $\Psi_\ell:=\bigcup_{j<\ell}\!\Psi_j\,$ $n$-trivializes $\Phi\!\restriction\!\xi_\ell$ and the construction continues. For successor $\ell$ we invoke the following lemma, with $\xi=\xi_{\ell-1}$ and $\eta=\xi_\ell$; this will conclude the proof of Theorem \ref{goblot}.
\begin{lemma} Fix $n>m\geq 0$ and $\xi<\eta$ with $\text{cf}\,(\xi)<\omega_m$ and let $\Phi=\{\varphi_{\vec{\beta}}\,|\,\vec{\beta}\in [\eta]^n\}$ be $n$-trivial. Any $n$-trivialization $\Psi$ of $\Phi\!\restriction\!\xi$ extends to an $n$-trivialization $\Psi'$ of $\Phi$.
\end{lemma}
\begin{proof} We prove the lemma in the course of the above inductive argument; it is here, in other words, that we apply the hypothesis that Claim \ref{31} holds for all $k<m$.

Consider first the following economizing notations: For any $n>0$ and $\Phi=\{\varphi_{\vec{\beta}}\,|\,\vec{\beta}\in [\delta]^n\}$ and $\Psi=\{\psi_{\vec{\beta}}\,|\,\vec{\beta}\in [\delta]^n\}$, write $\Phi=^*\Psi$ if $\varphi_{\vec{\beta}}=^*\psi_{\vec{\beta}}$ for all $\vec{\beta}\in [\delta]^n$. Write $d\Psi$ for the family $\{\theta_{\vec{\gamma}}:=\sum_{i=0}^n(-1)^i\psi_{\vec{\gamma}^i}\,|\,\vec{\gamma}\in [\delta]^{n+1}\}$. Observe that $d(\Phi+\Psi)=d\Phi+d\Psi$, and that $d^2=^* 0$, that is, $d(d(\Phi))$ is a family of finitely supported functions. The expressions $d\Phi=^*0$ and $d\Psi=^*\Phi$ are then useful shorthands for ``$\Phi$ is $n$-coherent'' and ``$\Psi$ $n$-trivializes $\Phi$,'' respectively.

We return now to the statement of the lemma. We argue only the cases of $n>2$; modifications for smaller $n$ are cosmetic and left to the reader. By assumption, there exists an $\Upsilon$ with $d\Upsilon=^*\Phi$; hence $$d(\Upsilon\!\restriction\!\xi-\Psi)=^*\Phi\!\restriction\!\xi-\Phi\!\restriction\!\xi=0$$
In other words, $(\Upsilon\!\restriction\!\xi-\Psi)$ is $(n-1)$-coherent. Since $\text{cf}(\xi)=\omega_k$ for some $k<m\leq n-1$, the inductive hypothesis furnishes us with a $\Theta=\{\theta_{\vec{\beta}}\,|\,\vec{\beta}\in [\xi]^{n-2}\}$ for which $d\Theta =^*\Upsilon\!\restriction\!\xi-\Psi$. Extend $\Theta$ to any $\bar{\Theta}=\{\theta_{\vec{\beta}}\,|\,\vec{\beta}\in [\eta]^{n-2}\}$ and let $\bar{\Psi}=\Upsilon-d\bar{\Theta}$. Observe that \begin{enumerate}
\item $d\bar{\Psi}=d\Upsilon-d(d(\bar{\Theta}))=^*\Phi$, and
\item $\Psi=^*\bar{\Psi}\!\restriction\!\xi$.
\end{enumerate}
In consequence, item (1) will continue to hold for a modified $\bar{\Psi}$, namely one in which the functions indexed by $[\xi]^{n-1}$ are adjusted to equal those of $\Psi$. This modified $\bar{\Psi}$ is the desired extension $\Psi'$ of $\Psi$ which $n$-trivializes $\Phi$.
\end{proof}
\end{proof}
By the following, Theorem \ref{goblot} is sharp, in the sense that it cannot hold for any $\delta$ of cofinality $\omega_n$.
\begin{theorem}[\cite{BergfalkThesis}]\label{nonntrivncohomegan} For all $n\geq 0$, there exists a non-$n$-trivial $n$-coherent family of height $\omega_n$.
\end{theorem}
As noted above, this theorem derives ultimately from Barry Mitchell's argument in \cite{rings} that the homological dimension of $\omega_n$ is $n+1$. Together with Theorem \ref{goblot}, it describes $(n+1)$-dimensional combinatorial relations first manifesting at the ordinal $\omega_n$ in any model of the $\mathsf{ZFC}$ axioms. These principles are core topological invariants generalizing nontrivial coherence; in consequence, Theorem \ref{nonntrivncohomegan} suggests strong generalizations of Remark \ref{remark} to the family of all ordinals $\omega_n$ $(n\in\omega)$.

However, the derivation in \cite{BergfalkThesis} of Theorem \ref{nonntrivncohomegan} from \cite{rings} is unsatisfactory in several respects: it is elaborate and tedious, and it only deduces the existence of non-$n$-trivial $n$-coherent families of functions mapping to large abelian groups $A$. Nevertheless, in its course, higher-dimensional generalizations of the walks apparatus of Section \ref{I.1} are uncovered; these are the key to more refined results, and a main focus of the forthcoming \emph{Cohomology of Ordinals II: ZFC Results} \cite{CohomologyII}.

The following theorem summarizes our results so far, and orients the work of the next section:

\begin{theorem} For any ordinal $\delta$ and abelian group $A$,
\smallskip
\begin{equation*}
  \check{\mathrm{H}}^n(\delta,\mathcal{A}_d) =
  \begin{cases}
                                  \text{the group of} & \\[-.5 ex] \text{0-coherent functions }\delta\rightarrow A & \text{if }n=0 \\[1 ex]
                                   \text{the group of} & \\[-.5 ex]
                                   \text{$n$-coherent }\{\varphi_{\vec{\alpha}}:\alpha_0\rightarrow A\,|\,\vec{\alpha}\in [\delta]^n\} & \\[-.5 ex]
                                    \text{modulo the group of} & \\[-.5 ex]
                                   \text{$n$-trivial }\{\varphi_{\vec{\alpha}}:\alpha_0\rightarrow A\,|\,\vec{\alpha}\in [\delta]^n\} & \text{if }n>0 \\
  \end{cases}
\end{equation*}
If $\delta$ is of cofinality $\omega_k$, then $\check{\mathrm{H}}^n(\delta,\mathcal{A}_d)=0$ for all $n>k$. Moreover, there exists in this case a group $B$ such that $\check{\mathrm{H}}^k(\delta,\mathcal{B}_d)\neq 0$.
\end{theorem}

When $\delta$ is of cofinality $\omega_1$, we can say more:  in this case, $\check{\mathrm{H}}^1(\delta,\mathcal{A}_d)\neq 0$ for any nontrivial abelian group $A$. In addition, $\check{\mathrm{H}}^1(\delta,\mathcal{A})\neq 0$ for all nontrivial metrizable abelian groups $A$ --- for $A=\mathbb{Q}$, $\mathbb{R}$, $\mathbb{C}$, or the circle group $\mathbb{T}$, for example (see \cite{BergfalkThesis}). Moreover, for any sheaf $\mathcal{S}$ and topological space $X$ the first \v{C}ech and sheaf cohomology groups $\check{\mathrm{H}}^1(X,\mathcal{S})$ and $\mathrm{H}^1(X,\mathcal{S})$, respectively, are isomorphic (see \cite{Tohoku}). Hence many of the sheaves $\mathcal{A}$ most classically valued for the vanishing of their higher cohomology groups $\mathrm{H}^n(\,\cdot\,,\mathcal{A})$ $(n\geq 1)$ in paracompact settings are no longer acyclic in the setting of $\omega_1$.

\section{Consistency Results}\label{sectiontwo}

We begin this section with a useful lemma. It tells us that for any $n\geq 0$ and ordinal $\delta$ the group $\chm^n(\text{cf}(\delta), \mathcal{A}_d)$ is a subgroup (and in fact a retract) of $\chm^n(\delta, \mathcal{A}_d)$. This will motivate our focus for the remainder of the section on the \v{C}ech cohomology groups of regular cardinals.

\begin{lemma}\label{iff} Suppose that $n\geq 0$, that $A$ is an abelian group, and that $C_\varepsilon=\linebreak \langle\,\xi_\alpha\,|\,\alpha\in\delta\,\rangle$ is a closed cofinal subset of $\varepsilon$. Then $C_\varepsilon$ determines an injective homomorphism of $\chm^n(\delta, \mathcal{A}_d)$ into $\chm^n(\varepsilon, \mathcal{A}_d)$. Relatedly, $\chm^n(\delta, \mathcal{B}_d)=0$ for all abelian groups $B$ if and only if $\chm^n(\varepsilon, \mathcal{B}_d)=0$ for all abelian groups $B$. \end{lemma}
\begin{proof} We begin with the first assertion, from which the ``if'' direction of the second assertion follows immediately. The first assertion's essential content is the following: in its ``stretching'' of $\delta$ to $\varepsilon$, the club $C_\varepsilon$ stretches height-$\delta$ non-$n$-trivial $n$-coherent families of functions to height-$\varepsilon$ non-$n$-trivial $n$-coherent families of functions as well. The argument of Corollary \ref{DA} exemplified the process in the case of $n=1$: derive a height-$\varepsilon$ family $\Phi'$ from a $\Phi=\{\varphi_\beta\,|\,\beta<\delta\}$ by letting $\varphi_{\xi_\beta}(\xi_\alpha)=\varphi_\beta(\alpha)$ whenever $\alpha<\beta$ and then extending this definition in the most obvious way. The extension of this procedure to other $n\geq 0$ and the verification that it transmits non-$n$-trivial $n$-coherence are straightforward but tedious, and left to the interested reader.

The following abstraction of the above argument is conceptually cleaner but requires machinery. We follow \cite{strongshape} in our treatment of the category $\mathsf{inv}$-$\mathrm{Ab}$ of inverse systems of abelian groups, the category $\mathsf{pro}$-$\mathrm{Ab}$, and the actions of $\text{lim}^n$ on each. For readers familiar with this machinery, the argument should be straightforward; we therefore only sketch the proof. Let $\mathbf{Q}(\delta)$ denote the inverse system $(Q_\beta,q_{\beta\gamma},\delta)$, where each $Q_\beta=\bigoplus_\beta A$ and each $q_{\beta\gamma}:Q_\gamma\to Q_\beta$ $(\beta<\gamma<\delta)$ is the restriction map. Then $\alpha\mapsto\xi_\alpha$ defines a morphism $\mathbf{i}:\mathbf{Q}(\delta)\to\mathbf{Q}(\varepsilon)$ as follows: view the elements of any $Q_\beta$ as functions $f:\beta\to A$. The morphism $\mathbf{i}$ is the family of maps $i_{(\beta,\gamma)}:Q_\beta\to Q_\gamma$ indexed by $\sigma=\{(\beta,\gamma)\,|\,\beta<\delta\text{ and }\xi_\beta\geq \gamma>\xi_\alpha\text{ for all }\alpha<\beta\}$ such that \begin{equation*}
  i_{(\beta,\gamma)}(f)(\xi) =
  \begin{cases}
                                  f(\alpha) & \text{if }\xi=\xi_\alpha \\
                                  0 & \text{if }\xi\in\gamma\backslash C_\varepsilon
  \end{cases}
\end{equation*}
The reader may easily supply a left-inverse $\mathbf{r}$ to the morphism $\mathbf{i}$. It is shown in \cite{BergfalkThesis} that $\lim^n\mathbf{Q}(\delta)$ is naturally isomorphic to $\mathrm{H}^n(\mathcal{U}_\delta,\mathcal{D}_A)$. It then follows from Theorem \ref{combicasting} together with the fact that $\text{lim}^n$ is functorial on $\mathsf{inv}$-$\mathrm{Ab}$ that the identity map $\mathbf{r}\circ\mathbf{i}:\mathbf{Q}(\delta)\to\mathbf{Q}(\varepsilon)\to\mathbf{Q}(\delta)$ descends to an identity homomorphism $\chm^n(\delta, \mathcal{A}_d)\to\chm^n(\varepsilon, \mathcal{A}_d)\to\chm^n(\delta, \mathcal{A}_d)$ factoring through $\chm^n(\varepsilon, \mathcal{A}_d)$. This shows the first assertion.

It remains only to show that if $\chm^n(\varepsilon, \mathcal{A}_d)\neq 0$ for some abelian group $A$ then $\chm^n(\delta, \mathcal{B}_d)\neq 0$ for some abelian group $B$. Let $\mathbf{P}(\delta)=(P_\beta,p_{\beta\gamma},\delta)$ where $P_\beta=\bigoplus_{\xi_\beta} A$ and each $p_{\beta\gamma}:P_\gamma\to P_\beta$ is the restriction map, as before. Then the family $f_\beta=\mathrm{id}:Q_{\xi_\beta}\to P_\beta$ $(\beta<\delta)$ defines a morphism $\mathbf{f}: \mathbf{Q}(\varepsilon)\to\mathbf{P}(\delta)$ in $\mathsf{inv}$-$\mathrm{Ab}$. The family of restriction maps $g_{(\beta,\gamma)}: P_\beta\to Q_\gamma$ $((\beta,\gamma)\in\sigma)$ defines a $\mathbf{g}:\mathbf{P}(\delta)\to\mathbf{Q}(\varepsilon)$ such that $\mathbf{g}\circ\mathbf{f}$, though not an isomorphism in $\mathsf{inv}$-$\mathrm{Ab}$, is an isomorphism in the category $\mathsf{pro}$-$\mathrm{Ab}$. As above, apply the functoriality of $\text{lim}^n$ on $\mathsf{pro}$-$\mathrm{Ab}$ to deduce the nontriviality of $\text{lim}^n \mathbf{P}(\delta)$ from that of $\text{lim}^n \mathbf{Q}(\varepsilon)$. Let $\mathbf{Q}^B(\delta)$ denote the inverse system defined like $\mathbf{Q}(\delta)$ but with $B$ in place of $A$. Take then $B=\bigoplus_\eta A$ with $\eta$ sufficiently large that $\mathbf{P}(\delta)$ appears naturally as a subsystem, and indeed as a retract, of $\mathbf{Q}^B(\delta)$. Then argue as in the previous paragraph the nontriviality of $\text{lim}^n \mathbf{Q}(\varepsilon)$, and hence of $\text{lim}^n \mathbf{P}(\delta)$, and hence of $\text{lim}^n \mathbf{Q}^B(\delta)$, and hence of $\chm^n(\delta, \mathcal{B}_d)$, from the nontriviality of $\chm^n(\varepsilon, \mathcal{A}_d)$.
\end{proof}

As the above argument suggests, our overwhelming interest in this section will simply be the question of whether or not the \v{C}ech cohomology groups $\chm^n(\varepsilon, \mathcal{A}_d)$ vanish. Though coarse, this is standard; as Pierce notes in related contexts, ``It is unlikely that [the question of the structure of the abelian groups $\text{Ext}^n(\,\cdot\,,\,\cdot\,)$ even in fairly restricted settings] has a reasonable solution'' \cite[p. 336]{pierce}. Nevertheless, some more refined questions, beginning, e.g., with the question of whether $\chm^1(\omega_1, \mathbb{Z}_d)$ may be free (due to Viale), do seem to us worthwhile; they simply fall beyond the scope of this paper.

\subsection{Compact cardinals}\label{II.1}\hfill

Recall the following characterization of a weakly compact cardinal.

\begin{definition} 
	A cardinal $\kappa$ is \textit{$\Pi_m^n$-indescribable} if whenever $U\subset V_\kappa$ and $\sigma$ is a $\Pi_m^n$ 
	sentence such that $(V_\kappa,\in,U)\vDash\sigma$, then for some $\alpha<\kappa$, $(V_\alpha,\in,U\cap V_\alpha)\vDash \sigma$.
\end{definition}

\begin{theorem}[Hanf-Scott; see \cite{kanamori}] 
	A cardinal $\kappa$ is $\Pi_1^1$-indescribable if and only if it is weakly compact.
\end{theorem}

\begin{theorem}\label{nowklycpt} 
	Suppose that $\kappa$ is a weakly compact cardinal, $n$ is a positive integer, and $A$ is an abelian group of 
	cardinality less than $\kappa$. Then every $n$-coherent family of $A$-valued functions indexed 
	by $\kappa$ is $n$-trivial. In other words, $\chm^n(\kappa, \mathcal{A}_d) = 0$.
\end{theorem}

\begin{proof}
	Since $|A| < \kappa$, we can assume that $A \in V_\kappa$. Suppose for sake of contradiction that 
	$\Phi = \{\varphi_{\vec{\alpha}}:\alpha_0 \rightarrow A \mid \vec{\alpha} \in [\kappa]^n\}$ is 
	$n$-coherent and non-$n$-trivial. Then $(V_\kappa,\in,\Phi)$ satisfies the following $\Pi_1^1$-sentence:
	\begin{align*}\tag{$\sigma$} 
		\forall\,\{\psi_{\vec{\beta}}:\beta_0\rightarrow A\,|\,\vec{\beta}\in[\text{Ord}]^{n-1}\}\,
		\exists\,\vec{\alpha}\in[\text{Ord}]^n\,\left[\sum_{i=0}^{n-1}(-1)^i\psi_{\vec{\alpha}^i}\neq^* \varphi_{\vec{\alpha}}\,\right]
	\end{align*}
	(We have preferred a somewhat more readable statement of $\sigma$ to a more formally correct one.) 
	By the $\Pi^1_1$-indescribability of $\kappa$, then, there is $\alpha < \kappa$ such that $A \in V_\alpha$ and
	$(V_\alpha, \in, \Phi \cap V_\alpha)$ satisfies $\sigma$. But this implies that $\Phi \cap V_\alpha$ is 
	non-$n$-trivial and therefore extends to no larger $n$-coherent family, which, by the discussion around 
	Example~\ref{elpmaxe}, contradicts the fact that $\Phi \cap V_\alpha$ is a strict initial segment of $\Phi$.
\end{proof}

\begin{remark} \label{weakly_compact_remark}
	In Theorem~\ref{nowklycpt}, if $n = 1$, then the requirement $|A| < \kappa$ can be dropped. 
	Suppose otherwise.
	We can always assume that $|A| \leq \kappa$, since at most $\kappa$-many elements of $A$ 
	appear as values in the family of functions. If $|A| = \kappa$, we can code $A$ as a subset of $V_\kappa$, and the sentence 
	$(\sigma)$ is satisfied by the structure $(V_\kappa, \in, \{\Phi, A\})$ and hence is also 
	satisfied by $(V_\alpha, \in, \{\Phi \cap V_\alpha, A \cap V_\alpha\})$ for some $\alpha < \kappa$. 
	In this case, though, the $(\Phi \cap V_\alpha)$-trivializing function $\varphi_\alpha$ takes values in $A \cap V_\alpha$ for all but finitely many elements of 
	its domain, which leads to a contradiction just as above.
\end{remark}

To obtain more global results, we will need to introduce more substantial large cardinal 
notions.

\begin{definition}
	A cardinal $\kappa$ is \emph{strongly compact} if, for every set $X$, every $\kappa$-complete filter on 
	$X$ extends to a $\kappa$-complete ultrafilter on $X$.
\end{definition}

The following weakening of strong compactness was introduced by Bagaria and Magidor in \cite{bagaria_magidor}.

\begin{definition} 
	Suppose that $\mu \leq \kappa$ are uncountable cardinals. Then $\kappa$ is \emph{$\mu$-strongly compact} 
	if, for every set $X$, every $\kappa$-complete filter on $X$ extends to a $\mu$-complete ultrafilter 
	on $X$.
\end{definition}

\begin{remark}
	Notice that if $\kappa$ is $\mu$-strongly compact and $\lambda > \kappa$ then $\lambda$ is also 
	$\mu$-strongly compact. Bagaria and Magidor showed in \cite{bagaria_magidor} that, for a fixed $\mu$, the least 
	$\mu$-strongly compact cardinal need not be regular, though it must be a limit cardinal whose cofinality 
	is at least the first measurable cardinal.
\end{remark}

\begin{theorem} \label{strongly_compact_thm}
	Suppose that $A$ is an abelian group, $\kappa \geq \mu > |A|$ are uncountable cardinals, and $\kappa$ 
	is $\mu$-strongly compact. Then $\chm^n(\lambda, \mathcal{A}_d) = 0$ for every $n \geq 1$ and every 
	regular cardinal $\lambda \geq \kappa$.
\end{theorem}

\begin{proof} 
	Again we will focus on $n \geq 2$, as the case $n=1$ is similar and easier. 
	Fix $n>1$ and $\lambda \geq \kappa$ as above. Notice that the filter 
	$\mathcal{F} = \{X \subseteq \lambda \mid |\lambda \setminus X| < \lambda\}$ on $\lambda$ is 
	$\kappa$-complete, so we can extend it to a $\mu$-complete ultrafilter $\mathcal{U}$ on 
	$\lambda$. Let $\Phi=\{\varphi_{\vec{\gamma}}:\gamma_0\rightarrow A\,|\,\vec{\gamma}\in [\lambda]^n\}$ 
	be $n$-coherent. We will show that $\Phi$ is also $n$-trivial by constructing a trivializing family 
	$\Psi$.

	For $\vec{\beta}\in [\lambda]^{n-1}$ and $\xi<\beta_0$ let $\psi_{\vec{\beta}}(\xi)$ be the unique $a\in A$ 
	such that 
	\[
	U_{\vec{\beta},\xi}:=\{\gamma\in(\beta_{n-2},\lambda)\,|\,\varphi_{\vec{\beta}\gamma}(\xi)=a\}\in\mathcal{U}.
	\] 
	Such an $a$ exists because $\mathcal{U}$ is $\mu$-complete and $|A| < \mu$.
	We claim that $\Psi=\{(-1)^n\psi_{\vec{\beta}}\,|\,\vec{\beta}\in [\lambda]^{n-1}\}$ $n$-trivializes $\Phi$. 
	If not, then for some $\vec{\gamma}\in [\lambda]^n$ and $x\in [\gamma_0]^{\aleph_0}$, 
	\begin{align}\label{xerror} 
		\sum_{i=0}^{n-1}(-1)^{i+n}\psi_{\vec{\gamma}^i}(\xi)\neq \varphi_{\vec{\gamma}}(\xi)\text{ for all }\xi\in x.
	\end{align} 
	Now let $I$ be the index set $(n-1) \times x$, and take some 
	\[
		\delta\in\bigcap_{(i,\xi)\in I} U_{\vec{\gamma}^i\!,\,\xi}\backslash\gamma_{n-1}
	\]
	By (\ref{xerror}), and the definition of $\Psi$,
	\[
		\sum_{i=0}^{n-1}(-1)^{i+n}\varphi_{\vec{\gamma}^i\delta}(\xi)\neq \varphi_{\vec{\gamma}}(\xi)\text{ for all }\xi\in x.
	\]
	Rearranging terms, this yields 
	\[
		(-1)^{n-1}\varphi_{\vec{\gamma}}(\xi) + \sum_{i=0}^{n-1}(-1)^i \varphi_{\vec{\gamma}^i\delta}(\xi) \neq 0 \text{ for all } \xi \in x.
	\]
	However, this contradicts the $n$-coherence of $\Phi$ applied to the set $\vec{\gamma} \cup \{\delta\} \in [\lambda]^{n+1}$. 
	Hence $\Phi$ is $n$-trivial.
\end{proof}

\subsection{Trivial first cohomology}\label{II.2}\hfill

In this subsection, we present two methods for obtaining trivial first cohomology groups at all regular 
cardinals greater than $\aleph_1$. We begin with the following useful lemma.

\begin{lemma} \label{one_branch_lemma}
	Suppose that $\delta$ is an ordinal of uncountable cofinality, $A$ is an abelian group, 
	$\Phi = \{\varphi_\alpha: \alpha \rightarrow A \mid \alpha < \delta\}$ is a nontrivial coherent family of functions, and $\bb{P}$ is a forcing notion such that 
	$\Vdash_{\bb{P} \times \bb{P}}``\cf(\delta) > \omega"$. Then 
	$\Phi$ remains nontrivial after forcing with $\bb{P}$.
\end{lemma}

\begin{proof}
	Suppose to the contrary that a condition $p \in \bb{P}$ forces that some 
	$\dot{\psi}:\delta\to A$ trivializes 
	$\Phi$. Let $G \times H$ be a $\bb{P} \times \bb{P}$-generic ultrafilter containing $(p,p)$ 
	and let $\psi_G$ and $\psi_H$ be the realizations of $\dot{\psi}$ in $V[G]$ and $V[H]$, 
	respectively. Observe that $(p,p)\in G \times H$ implies that each of the functions $\psi_G$ and 
	$\psi_H$ trivializes $\Phi$. 
	On the other hand, it follows from the product lemma for forcing and the fact that $\Phi$ is 
	nontrivial in $V$ that $\psi_G \notin V[H]$ and $\psi_H \notin V[G]$; hence there exists in $V[G \times H]$ some infinite set $X \subseteq \delta$ such that 
	$\psi_G(\alpha) \neq \psi_H(\alpha)$ for every $\alpha \in X$. By our assumption on $\mathbb{P}$, there also exists in $V[G \times H]$ a $\beta < \delta$ such that $X \cap \beta$ is infinite. 
	By definition, $\psi_G \!\restriction\! \beta =^* \varphi_\beta =^* \psi_H\! \restriction \!\beta$, 
	but this now contradicts our deduction that $\psi_G$ and $\psi_H$ differ at each of the infinitely many arguments $\alpha$ in $X\cap\beta$.
\end{proof}

\begin{corollary} \label{countably_closed_cor}
	Suppose that $\delta$ is an ordinal of uncountable cofinality, $A$ is an abelian group, 
	$\Phi = \{\varphi_\alpha:\alpha \rightarrow A \mid \alpha < \delta\}$ is a nontrivial coherent family of functions, and $\bb{P}$ is a countably closed forcing notion. Then 
	$\Phi$ remains nontrivial after forcing with $\bb{P}$.
\end{corollary}

\begin{proof}
	Since $\bb{P}$ is countably closed, $\bb{P} \times \bb{P}$ is also countably 
	closed and therefore does not add any $\omega$-sequences of ordinals. The result 
	now follows from Lemma~\ref{one_branch_lemma}.
\end{proof}
In the argument below, we follow the essentially standard notations of \cite{cummings_handbook}.
\begin{theorem}\label{stronglycmpct}
	Suppose that $\mu$ is a regular uncountable cardinal, $\kappa > \mu$ is a strongly compact cardinal, and 
	$\bb{P} = \mathrm{Coll}(\mu, <\! \kappa)$. Then for 
	every regular cardinal $\lambda \geq \kappa$ and every abelian group $A$ in the forcing extension of $V$ by $\mathbb{P}$, we have $\chm^1(\lambda, \mathcal{A}_d) = 0$.
\end{theorem}

\begin{proof} 	Let $G$ be $\bb{P}$-generic over $V$.
	Fix in $V$ a regular cardinal $\lambda \geq \kappa$. Since $\kappa$ is strongly compact, there exists an 
	elementary embedding $j:V \rightarrow M$ with $\mathrm{crit}(j) = \kappa$ together with an 
	$X \in M$ such that $j``\lambda \subseteq X$ and $M \models ``|X| < j(\kappa)"$. In particular, 
	$\sup(j``\lambda) < j(\lambda)$. Let $\eta = \sup(j``\lambda)$. 

	Consider now an abelian group $A$ and coherent family $\Phi = \{\varphi_\alpha:\alpha \rightarrow A \mid 
	\alpha < \lambda\}$ in $V[G]$. Let $\dot{\Phi}$ and $\dot{A}$ be $\mathbb{P}$-names for these objects. 
	Observe that the forcing $j(\bb{P})/G = \mathrm{Coll}(\mu, <\! j(\kappa))/G$ is countably closed, and let $H$ be $j(\bb{P})/G$-generic over $V[G]$. By standard arguments, the extension of the $\bb{P}$-generic $G$ to a $j(\bb{P})$-generic $G*H$ induces an extension of the elementary embedding $j:V\to M$ to an elementary 
	embedding $k:V[G] \rightarrow M[G * H]$ existing in $V[G*H]$. In particular, $k(\Phi)$ is a coherent family in $M[G*H]$ indexed by $k(\lambda)$. 
	We denote it $\Psi = \{\psi_\gamma:\gamma \rightarrow k(A) \mid \gamma < k(\lambda)\}$. 

	Observe that $k(\varphi_\alpha(\xi)) = \psi_{k(\alpha)}(k(\xi))$ for all $\xi < \alpha < \lambda$.
	Hence by the coherence of $\Psi$, the equality $\psi_\eta(k(\xi)) = k(\varphi_\alpha(\xi))$ holds for all $\alpha < \lambda$ 
	and all but finitely many $\xi < \alpha$. In particular, 
	$\psi_\eta(k(\xi)) \in k``A$ for all but finitely many $\xi < \lambda$. Define a function 
	$\varphi:\lambda \rightarrow A$ as follows:
	\begin{equation*}
  \varphi(\xi) =
  \begin{cases}
                                  k^{-1}(\psi_\eta(k(\xi))) & \text{if $\psi_\eta(k(\xi))\in k``A$} \\
                                   0 & \text{otherwise}.
  \end{cases}
\end{equation*}
The function $\varphi\in V[G*H]$ trivializes $\Phi$. Since $V[G*H]$ is an extension of 
	$V[G]$ by the countably closed forcing $j(\bb{P})/G$, Corollary~\ref{countably_closed_cor} applies, and
	$\Phi$ must already have been trivial in $V[G]$. Since $\lambda$ and $A$ were arbitrary, this completes the proof 
	of the theorem.
\end{proof}

\begin{corollary}
It is consistent relative to the consistency of the existence of a strongly compact cardinal that $\chm^1(\lambda, \mathcal{A}_d) = 0$ for every abelian group $A$ and regular cardinal $\lambda\geq\aleph_2$.
\end{corollary}

\begin{remark}\label{local}
	Using similar methods, one can prove a more local result starting with smaller large cardinals. More specifically, 
	if $\mu$ is a regular uncountable cardinal and $\kappa > \mu$ is a weakly compact cardinal then after forcing 
	with $\mathrm{Coll}(\mu, <\!\kappa)$ one obtains a model in which $\kappa = \mu^+$ and $\chm^1(\kappa, \mathcal{A}_d) = 0$ 
	for every abelian group $A$. Details are left to the reader.
	
	Another obvious path to local results is to invoke the tree property: reasoning exactly as for Corollary \ref{aronszajn} shows that if $\kappa$ has the tree property then $\chm^1(\kappa, \mathcal{A}_d) = 0$ for any $A$ of cardinality less than $\kappa$. More careful reasoning will often show more; in the Mitchell model \cite{mitchell} of the tree property at $\aleph_2$, for example, $\chm^1(\aleph_2, \mathcal{A}_d) = 0$ for all abelian groups $A$, as the reader again may verify.
	
	This might be contrasted with the model of Theorem \ref{stronglycmpct} above, when $\mu=\aleph_1$ and the strongly compact $\kappa$ is consequently collapsed to $\aleph_2$. Therein, $\aleph_2$-Aronsjazn trees do exist --- only none of them are \emph{coherent}, in the natural sense, e.g., of \cite{lipschitz}. 
\end{remark}

For global results, one might alternatively apply the \emph{P-Ideal Dichotomy} of \cite{pid}:

\begin{definition} Let $X$ be a set. A \emph{P-ideal on $X$} is an ideal $\mathcal{I}\subseteq [X]^{\leq\aleph_0}$ such that
\begin{enumerate}
\item $\mathcal{I}$ contains all finite subsets of $X$.
\item For every $\{x_n\,|\,n\in\omega\}\subseteq\mathcal{I}$ there exists an $x\in\mathcal{I}$ satisfying $x_n\subseteq^* x$ for all $n$.
\end{enumerate}
The \emph{P-Ideal Dichotomy} is the following assertion:

If $\mathcal{I}$ is a P-ideal on any set $X$, then exactly one of the following alternatives holds:
\begin{enumerate}
\item There exists an uncountable $B\subseteq X$ such that $[B]^{\leq\aleph_0}\subseteq\mathcal{I}$.
\item $X=\bigcup_{i\in\omega}B_i$ with $[B_i]^{\aleph_0}\cap\mathcal{I}=\varnothing$ for each $i$.
\end{enumerate}
\end{definition}

The following theorem is due to Stevo Todorcevic (\cite{StevoPIDtip}), but its proof has not appeared in print.

\begin{theorem}[Todorcevic] \label{pidtheorem} Assume the P-Ideal Dichotomy, and let $A$ be a nontrivial abelian group. Then $\check{\mathrm{H}}^1(\varepsilon,\mathcal{A}_d)\neq 0$ if and only if $\text{cf}(\varepsilon)=\omega_1$.
\end{theorem}

\begin{proof} The ``if'' implication is immediate from Corollary \ref{DA} and Theorem \ref{combicasting}. For the implication in the other direction, fix an ordinal $\varepsilon$ of cofinality greater than $\omega_1$. By Lemma \ref{iff}, we may without loss of generality assume $\varepsilon$ to be regular. Now fix a coherent family of functions $\Phi=\{\varphi_\alpha:\alpha\rightarrow A\,|\,\alpha<\varepsilon\}$. For any $\alpha<\beta$ in $\varepsilon$ let
$$e(\alpha,\beta)=\{\xi<\alpha\,|\,\varphi_\alpha(\xi)\neq\varphi_\beta(\xi)\}$$ and let $$\mathcal{I}=\{ b\in[\varepsilon]^{\leq\aleph_0}\,|\,\exists\,\beta\geq\sup(b)\,\forall\,n\in\omega\,\{\alpha\in b\,|\;|e(\alpha,\beta)|\leq n\}\text{ is finite}\}$$

\begin{claim} $\mathcal{I}$ is a P-ideal of countable subsets of $\varepsilon$.
\end{claim}
\begin{proof} For $b_i\in\mathcal{I}$ $(i\in\omega)$, take $\gamma\geq\sup(\cup_{i\in\omega}b_i)$, and for $n\in\omega$ let $$b_i(n,\gamma)=\{\beta\in b_i\,|\;|e(\beta,\gamma)|=n\}$$By assumption, each $b_i(n,\gamma)$ is finite; in consequence, $\gamma$ witnesses that $$b:=\cup\{b_i(n,\gamma)\,|\,n>i\}\text{ is in }\mathcal{I}$$ Clearly $b_i\subseteq^* b$ for all $i\in\omega$.
\end{proof}
Assuming the P-Ideal Dichotomy, there are now two possibilities.

\textbf{Case 1}: In the first possibility, there exists a $B\in [\varepsilon]^{\aleph_1}$ with $[B]^{\leq\aleph_0}\subseteq\mathcal{I}$. Write $B$ as an increasing union of $b_i\in [\varepsilon]^{\leq\aleph_0}$ $(i\in\omega_1)$. For $\gamma\geq\sup(B)$ there exists some $m\in\omega$ and increasing $\{\alpha_\xi\,|\,\xi\in\omega_1\}\subseteq B$ such that $|e(\alpha_\xi,\gamma)|=m$ for all $\xi\in\omega_1$. Fix a $b_i$ containing $\{\alpha_\xi\,|\,\xi\in\omega\}$. Then there exists a $\beta\geq\sup(b_i)$ witnessing that $b_i\in\mathcal{I}$. Therefore:
$$\text{For all }n\in\omega\text{ there exists an }\ell(n)\text{ such that }\xi>\ell(n)\Rightarrow|e(\alpha_\xi,\beta)|>m+n$$
As $|e(\alpha_\xi,\gamma)|=m$, though, this implies that $|e(\beta,\gamma)|>n$ for all $n\in\omega$. This contradicts the premise that $\Phi$ is coherent.

\textbf{Case 2}: The second possibility, which by the above argument must hold, is that $\varepsilon=\bigcup_{i\in\omega}B_i$, with no infinite subset of any $B_i$ an element of $\mathcal{I}$. In consequence, there exists some stationary $B=B_i\subseteq\varepsilon$ such that for all $\gamma\in B$,$$n_\gamma^B=\max\{|e(\beta,\gamma)|\;|\,\beta\in B\cap\gamma\}$$ is well-defined.
\begin{claim}\label{thinnthin} For some stationary subset $S\subseteq B$, there exists an $N\in\omega$ such that\begin{enumerate}
\item $n_\gamma^S=\max\{|e(\beta,\gamma)|\;|\,\beta\in S\cap\gamma\}\leq N$ for all $\gamma\in S$, and
\item $T=\{\gamma\in S\,|\,\sup\{\beta\in S\cap\gamma\,|\;|e(\beta,\gamma)|=N\}=\gamma\}$ is stationary in $\varepsilon$.
\end{enumerate}
\end{claim}
\begin{proof}
Thin $B$ first to a stationary $S$ satisfying (1). If (2) fails, then $$\gamma\mapsto\sup\{\beta\in S\cap\gamma\,|\;|e(\beta,\gamma)|=N\}$$ is a regressive function on some stationary $E\subseteq S$, and is therefore constantly $\beta$ on some stationary $S_1\subseteq E\backslash\beta$. Now there exists an $N_1<N$ such that $n_\gamma^{S_1}\leq N_1$ for all $\gamma\in S_1$. If $$T=\{\gamma\in S_1\,|\,\sup\{\beta\in S_1\cap\gamma\,|\;|e(\beta,\gamma)|=N\}=\gamma\}$$ is again nonstationary, we may repeat this process, defining in turn an $N_2<N_1$. This process, defining a strictly decreasing sequence of natural numbers, must at some finite stage $i$ terminate; at this point, $S=S_i$ is as claimed.
\end{proof}
We now argue by induction on $N=n$ that an $S$ as in Claim \ref{thinnthin} witnesses the triviality of $\Phi$. The case $N=0$ is clear: $\bigcup_{\gamma\in S}\,\varphi_\gamma$ trivializes $\Phi$. For the induction step, assume that any $S$ as in Claim \ref{thinnthin} with $N=m<n$ witnesses the triviality of $\Phi$, and consider an $S$ as in Claim \ref{thinnthin} with $N=n$. For all $\gamma\in T$, let $a_\gamma=\{\alpha\in S\,|\;|e(\alpha,\gamma)|=n\}$. Consider then any $\gamma\in T$ and $\delta>\gamma$ in $S$. If $e(\gamma,\delta)\neq\varnothing$, then $e(\alpha,\gamma)\cap e(\gamma,\delta)\neq\varnothing$ for all $\alpha\in a_\gamma\cap (\max(e(\gamma,\delta)),\gamma)$, since by assumption $|e(\alpha,\gamma)|=n$ and $n\geq|e(\alpha,\delta)|\geq |e(\alpha,\gamma)\,\triangle\, e(\gamma,\delta)|$. Therefore one of the two following possibilities holds:
\begin{enumerate}
\item For all $\gamma$ in some stationary $U\subseteq T$ there exists some $\beta_\gamma<\gamma$ such that $e(\alpha,\gamma)\cap\beta_\gamma\neq\varnothing$ for all $\alpha\in a_\gamma\cap (\beta_\gamma,\gamma)$. By the Pressing Down Lemma, $\beta_\gamma=\beta$ for all $\gamma$ in some stationary $V\subseteq U$. Claim \ref{thinnthin} then holds for some $N<n$ and stationary set $W\subseteq V$ for the coherent family $$\Phi\!\restriction\!(\beta,\gamma):=\{\varphi_\gamma\!\restriction\![\beta,\gamma)\,|\,\gamma\in(\beta,\varepsilon)\}.$$By the induction hypothesis, $\Phi\!\restriction\!(\beta,\gamma)$ is trivial; in consequence, $\Phi$ is, as well.
\item For all $\gamma$ in some stationary $U\subseteq T$, $$e(\gamma,\delta)=\varnothing\text{ for all }\delta\in U\backslash\gamma.$$
(This alternative follows from our argument above that if $e(\gamma,\delta)=\varnothing$ for \textit{any} $\delta\in S$ then there exists some $\beta_\gamma<\gamma$ as in item 1.) In this case, $\bigcup_{\gamma\in U}\varphi_\gamma$ trivializes $\Phi$.
\end{enumerate}
In either case, the coherent family $\Phi$ is trivial. This concludes the proof.\end{proof}

It is perhaps worth emphasizing that for all of the compactness or large cardinal properties holding at $\omega_2$ in any of the models under discussion above, height-$\omega_2$ non-2-trivial 2-coherent families of functions nevertheless continue to exist in each, by Theorem \ref{nonntrivncohomegan}.

\subsection{Cohomology in L}\label{II.3}\hfill

In this subsection we show that if $\mathrm{V} = \mathrm{L}$ then the extent of nontriviality in the \v{C}ech cohomology groups of the ordinals is maximal; put differently, in the model $\mathrm{L}$ the only groups $\chm^n(\xi, \mathcal{A}_d)$ that are trivial are those that are required to be so by Theorems \ref{goblot} and \ref{nowklycpt}. We first need a few definitions.
The following principle was isolated by Todorcevic \cite{pairs}, building on Jensen's work on 
combinatorial principles holding in $\mathrm{L}$ \cite{FineStructure}. In what follows, if $\beta$ is a limit ordinal 
and $C$ is a club in $\beta$, then $\mathrm{Lim}(C)$ denotes the set $\{\alpha \in C \mid \sup(C \cap \alpha) = \alpha\}$.

\begin{definition}
	Suppose that $\lambda$ is a regular uncountable cardinal. The principle $\square(\lambda)$ asserts the existence of a sequence 
	$\vec{C} = \langle C_\alpha \mid \alpha < \lambda \rangle$ such that
	\begin{enumerate}
		\item for every limit ordinal $\alpha < \lambda$, $C_\alpha$ is a closed cofinal subset of $\alpha$;
		\item for every limit ordinal $\alpha < \beta$, if $\alpha \in \mathrm{Lim}(C_\beta)$, 
		then $C_\beta \cap \alpha = C_\alpha$;
		\item there is no closed cofinal $D\subseteq\lambda$ such that 
		$D \cap \alpha = C_\alpha$ for every $\alpha \in \mathrm{Lim}(D)$.
	\end{enumerate}
\end{definition}

\begin{remark} \label{square_remark}
A first explicit application of the above principle was Todorcevic's observation \cite{pairs} that, when defined with respect to a $\square(\lambda)$-sequence $\vec{C}$, the $\rho_2$ function $[\lambda]^2\to\mathbb{Z}$ is nontrivial and coherent, in the sense of Theorem \ref{coherenceofrhos}. Via natural coding techniques together with arguments given above, we may phrase this result as follows:
\begin{theorem}[\cite{pairs}] \label{todorcevicsquare} Suppose $\square(\lambda)$ holds for some regular uncountable cardinal $\lambda$. Then $\chm^1(\lambda,\mathcal{A}_d)\neq 0$ for any nontrivial abelian group $A$.
\end{theorem}
The argument in \cite{pairs} for $\rho_2$ with arbitrary $\lambda$ and $A=\mathbb{Z}$ was much like that given for Theorem \ref{coherenceofrhos} above; a $\vec{C}$-sequence satisfying the condition $(\star)$ of Section \ref{I.1}, after all, is a witness to the principle $\square(\omega_1)$. Results like Corollary \ref{DA} are from this perspective instances of Theorem \ref{todorcevicsquare}, artifacts of the status of $\square(\omega_1)$ as a $\mathsf{ZFC}$ theorem. On the other hand, for $\lambda$ greater than $\omega_1$, Theorem \ref{todorcevicsquare} is in evident tension with the results of Section \ref{II.2}; the following theorem, also deduced in \cite{pairs} from \cite{FineStructure}, helps to motivate the large cardinal assumptions pervading that section:
\begin{theorem}[\cite{pairs}] Let $\lambda$ be a regular uncountable cardinal. If $\square(\lambda)$ fails, then $\lambda$ is weakly compact in $\mathrm{L}$.
\end{theorem}
\end{remark}
Like nontrivial coherence, the principle $\square(\lambda)$ describes relations of local agreement that cannot be globalized. Affinities like those recorded in Theorem \ref{todorcevicsquare} are, from this perspective, perhaps unsurprising. However, we will see in Section \ref{II.5} that $\square(\lambda)$ and height-$\lambda$ nontrivial coherence and several related incompactness principles are nevertheless each quite distinct.

The following is a useful strengthening of $\square(\lambda)$:

\begin{definition}
	Suppose that $\lambda$ is a regular uncountable cardinal and $S \subseteq \lambda$ is stationary. The principle 
	$\square(\lambda, S)$ asserts the existence of a $\square(\lambda)$-sequence $\vec{C}$ such that $\mathrm{Lim}(C_\beta) \cap S = \emptyset$ for 
	all limit ordinals $\beta < \lambda$.
\end{definition}

We will also need the following guessing principle, introduced by Jensen.

\begin{definition}
	Suppose that $\lambda$ is a regular uncountable cardinal and $S \subseteq \lambda$ is stationary. 
	The principle $\diamondsuit_\lambda(S)$ asserts the existence of a sequence $\langle X_\beta \mid 
	\beta \in S \rangle$ such that
	\begin{enumerate}
		\item $X_\beta \subseteq X$ for every $\beta \in S$;
		\item for every $X \subseteq \beta$, the set $\{\beta \in S \mid X_\beta = X \cap \beta\}$ is stationary 
		in $\lambda$.
	\end{enumerate}
\end{definition} 

\begin{remark} \label{diamond_remark}
	For any set $A$ with $|A| < \lambda$ and any positive integer $n$, the principle $\diamondsuit_\lambda(S)$ is equivalent to 
	the existence of a sequence $\langle F_\beta \mid \beta \in S \rangle$ such that
	\begin{enumerate}
		\item $F_\beta$ is of the form $\{f^\beta_{\vec{\alpha}}:\alpha_0 \rightarrow A 
		\mid \vec{\alpha} \in [\beta]^n\}$ for every $\beta \in S$;
		\item for every $\Phi = \{\varphi_{\vec{\alpha}}:\alpha_0 \rightarrow A \mid \vec{\alpha} \in [\lambda]^n\}$, 
		the set $\{\beta \in S \mid F_\beta = \Phi \!\restriction\! \beta\}$ is stationary in $\lambda$.
	\end{enumerate}
	As before, $\Phi \!\restriction\! \beta$ denotes the set $\{\phi_{\vec{\alpha}} \mid \vec{\alpha} \in [\beta]^n\}$.
\end{remark}

Our $\mathrm{V} = \mathrm{L}$ result follows from a more general ``stepping-up" theorem:

\begin{theorem} \label{steppingup}
	Suppose that $\kappa < \lambda$ are regular uncountable cardinals, that $n$ is a positive integer, and that $A$ is an abelian group. 
	Suppose moreover that
	\begin{enumerate}
		\item there is a non-$n$-trivial $n$-coherent family of $A$-valued functions indexed by $\kappa$;
		\item there is a stationary $S \subseteq S^\lambda_\kappa$ such that $\square(\lambda, S)$ and 
		$\diamondsuit_\lambda(S)$ both hold.
	\end{enumerate} 
	Then there is a non-$(n+1)$-trivial $(n+1)$-coherent family of $A$-valued functions indexed by $\lambda$.
\end{theorem}

\begin{proof}
	Let $\Psi = \{\psi_{\vec{\eta}}:\eta_0 \rightarrow A \mid \vec{\eta} \in [\kappa]^n\}$ be a 
	non-$n$-trivial $n$-coherent family. Since the functions in $\Psi$ take at most $\kappa$ many 
	values in $A$ and since these values generate a subgroup of $A$ of cardinality at most $\kappa$, 
	we may pass to this subgroup and assume that $|A| \leq\kappa< \lambda$. Also fix a stationary $S \subseteq S^\lambda_\kappa$ and a 
	$\diamondsuit_\lambda(S)$-sequence $\mathcal{F} = \langle F_\beta \mid \beta \in S \rangle$ as described in Remark~\ref{diamond_remark} (with respect to the parameters $A$ and $n$) and 
	a $\square(\lambda, S)$-sequence $\vec{C} = \langle C_\alpha \mid \alpha < \lambda \rangle$. 
	For each $\beta \in S$ let $F_\beta = \{f^\beta_{\vec{\alpha}}:\alpha_0 \rightarrow A \mid \vec{\alpha} \in [\beta]^n\}$.

	The argument proceeds as follows: at stages $\delta<\lambda$, we recursively construct the initial segments $\Phi\!\restriction\!(\delta+1)$ of a 
	non-($n+1)$-trivial $(n+1)$-coherent family 
	$\Phi=\{\varphi_{\vec{\beta}}:\beta_0\rightarrow A\,|\,\vec{\beta}\in[\lambda]^{n+1}\}$. 
	In the course of this construction, $\mathcal{F}$ conjoined with the non-$n$-trivial $n$-coherent family of our premise organizes a diagonalization against any potential $(n+1)$-trivialization of $\Phi$; the sequence $\vec{C}$ will bypass the indices $S$ of $\mathcal{F}$ to escort $(n+1)$-coherence relations through the uncountable-cofinality stages of the construction. More precisely, by maintaining the following condition at all stages $\delta < \lambda$ of our construction we will avert the danger of achieving non-$(n+1)$-triviality too early:
	\begin{align*}\tag{$\dagger$} 
		\varphi_{\vec{\alpha}\beta}=\varphi_{\vec{\alpha}\gamma}\text{ for all limit }\gamma\in(\delta+1)\backslash S
		\text{ and }\beta \in \mathrm{Lim}(C_\gamma)\text{ and }\vec{\alpha}\in[\beta]^n
	\end{align*}
	To sum all this up: at stage $\delta$ the construction 
	consists of extending $\Phi\!\restriction\!\delta$ to 
	$\Phi\!\restriction\!(\delta+1)$ by way of a family of functions 
	$\{\varphi_{\vec{\beta}\delta}:\beta_0\rightarrow A\,|\,\vec{\beta}\in[\delta]^n\}$ which preserves 
	$(\dagger)$, coheres with $\Phi\!\restriction\!\delta$, and incorporates any relevant information in 
	$\mathcal{F}$. How we simultaneously achieve these goals will depend on the status of $\delta$.

	\textbf{Case 1: $\delta$ is a successor ordinal}. In this case, any $(n+1)$-coherent extension of $\Phi\!\restriction\!\delta$ to $\Phi\!\restriction\!(\delta+1)$ will do. Such exist since $\Phi\!\restriction\!\delta$ is trivial, by Theorem \ref{goblot}.

	\textbf{Case 2: $\delta\in\text{Lim}\backslash S$, with $\delta>\sup(\mathrm{Lim}(C_\delta))$}. Let $\gamma=\sup(\mathrm{Lim}(C_\delta))$; 
	the condition $(\dagger)$ obliges the definition $\varphi_{\vec{\beta}\delta}=\varphi_{\vec{\beta}\gamma}$ for 
	all $\vec{\beta}\in [\gamma]^n$. More generally, for any $\vec{\beta}\in [\delta]^n$ and $\alpha\in \beta_0\cap\gamma$ define
\[ \varphi_{\vec{\beta}\delta}(\alpha)= \begin{cases} 
      0 & \text{if $\gamma$ is among the coordinates of $\vec{\beta}$} \\
      
      (-1)^s\varphi_{(\vec{\beta}\gamma)}(\alpha) & \text{otherwise}
   \end{cases}
\]
Here $(\vec{\beta}\gamma)$ denotes the increasing ordering of the coordinates of $\vec{\beta}$ together with $\gamma$, while $s$ counts how many such coordinates succeed $\gamma$ in this ordering. Extend this definition to arguments $\alpha\in [\gamma,\beta_0)$ as follows: observe that $\text{cf}(\delta)=\omega$, hence again by Theorem \ref{goblot} some $\Psi=\{\psi_{\vec{\beta}}\,|\,\vec{\beta}\in [\delta]^n\}$ $(n+1)$-trivializes $\Phi\!\restriction\!\delta$. Whenever $\gamma\leq\alpha<\beta_0$ let $\varphi_{\vec{\beta}\delta}(\alpha)=(-1)^n\psi_{\vec{\beta}}(\alpha)$. This defines an $(n+1)$-coherent $\Phi\!\restriction\!(\delta+1)$ satisfying $(\dagger)$.

	\textbf{Case 3: $\delta\in\mathrm{Lim}(\lambda)\setminus S$ and $\sup(\mathrm{Lim}(C_\delta))=\delta$}. For $\vec{\beta}\in [\delta]^n$ 
	let $\varphi_{\vec{\beta}\delta}=\varphi_{\vec{\beta}\gamma}$ where $\gamma=\min(\mathrm{Lim}(C_\delta)\setminus(\beta_{n-1}+1))$. 
	Then $\Phi\!\restriction\!(\delta+1)$ is $(n+1)$-coherent, and $(\dagger)$ is maintained.

	\textbf{Case 4: $\delta\in S$ and $F_\delta$ does not $(n+1)$-trivialize $\Phi \restriction \delta$}. Proceed as in Case 3.

	\textbf{Case 5: $\delta\in S$ and $F_\delta$ does $(n+1)$-trivialize $\Phi \restriction \delta$}. By hypothesis (1) of our theorem 
	together with Lemma~\ref{iff} there exists a non-$n$-trivial $n$-coherent family 
	$\Psi^\delta=\{\psi^\delta_{\vec{\beta}}:\beta_0\rightarrow A\,|\,\vec{\beta}\in [\delta]^n\}$. Let 
	$\varphi_{\vec{\beta}\delta}=(-1)^nf^\delta_{\vec{\beta}}+\psi^\delta_{\vec{\beta}}$ for all $\vec{\beta}\in [\delta]^n$. 
	As the reader may verify, $\Phi\!\restriction\!(\delta+1)$ is $(n+1)$-coherent. As $\delta$ is in $S$, the family $\Phi\!\restriction\!(\delta+1)$ satisfies $(\dagger)$ simply because $\Phi\!\restriction\!\delta$ does.

	Proceeding in this way through stages $\delta<\lambda$, we construct an $(n+1)$-coherent family $\Phi$. Now suppose 
	towards contradiction that $\Upsilon = \{\upsilon_{\vec{\beta}}:\beta_0 \rightarrow A \mid \vec{\beta} \in [\lambda]^n\}$ 
	$(n+1)$-trivializes $\Phi$. Then $\Upsilon\!\restriction\!\delta=F_\delta$ 
	for some $\delta \in S$, hence stage $\delta$ fell in Case 5 of our construction. If we let 
	$\theta_{\vec{\beta}}=\upsilon_{\vec{\beta}\delta}$ for all $\vec{\beta}\in [\delta]^{n-1}$ 
	then the family $\{\theta_{\vec{\beta}}\,|\,\vec{\beta}\in [\delta]^{n-1}\}$ $n$-trivializes the family $\Psi^\delta$: for all $\vec{\beta}\in [\delta]^n$
	\[
		\varphi_{\vec{\beta}\delta}=\sum_{i=0}^n(-1)^i\upsilon_{(\vec{\beta}\delta)^i}=(-1)^n 
		f^\delta_{\vec{\beta}}+\sum_{i=0}^{n-1}(-1)^i\theta_{\vec{\beta}^i}=
		\varphi_{\vec{\beta}\delta}-\psi^\delta_{\vec{\beta}}+\sum_{i=0}^{n-1}(-1)^i\theta_{\vec{\beta}^i}
	\] 
	hence for all $\vec{\beta}\in [\delta]^n$ 
	\[
		\psi^\delta_{\vec{\beta}}=\sum_{i=0}^{n-1}(-1)^i\theta_{\vec{\beta}^i}.
	\]
	This contradicts the non-$n$-triviality of $\Psi^\delta$. Hence $\Phi$ is non-$(n+1)$-trivial.
\end{proof}

\begin{corollary}\label{L}
	Suppose that $\mathrm{V} = \mathrm{L}$. Suppose moreover that $n \geq 1$, $A$ is a non-trivial abelian group, 
	and $\lambda \geq \aleph_n$ is a regular cardinal that is not weakly compact. Then 
	$\chm^n(\lambda, \mathcal{A}_d) \neq 0$.
\end{corollary}

\begin{proof}
	Work in $\mathrm{L}$. By results of Jensen (see \cite{FineStructure}), for every pair of 
	regular uncountable cardinals $\kappa < \lambda$ neither of which are weakly compact, there exists a stationary 
	$S \subseteq S^\lambda_\kappa$ for which $\square(\lambda, S)$ and $\diamondsuit_S(\lambda)$ both hold. 
	 By Theorem \ref{todorcevicsquare}, for any nontrivial abelian group $A$ and $\lambda\geq\aleph_2$ that is not weakly compact, there exists a $\kappa<\lambda$ for which $\chm^1(\kappa,\mathcal{A}_d)\neq 0$. Hence $\chm^2(\lambda,\mathcal{A}_d)\neq 0$, by Theorem~\ref{steppingup}. 
	Re-applying Theorem~\ref{steppingup} then secures the corollary for $n=3$, then $n=4$, and so on.
\end{proof}

\subsection{Forcing nontrivial coherence}\label{II.4}\hfill

In this subsection we show how to add non-$n$-trivial $n$-coherent families of functions by forcing 
and how to then $n$-trivialize these families via further forcing. Readers aware of forcings 
to add and thread various square sequences will find much that is familiar here.

\begin{definition}
	Suppose that $\delta$ is an ordinal and $\bb{P}$ is a forcing poset with a maximal element $1_\bb{P}$. 
	\begin{enumerate}
		\item The game $\Game_\delta(\bb{P})$ is a two-player game of (potential) length $\delta$ in which Players $\mathsf{Odd}$ and 
		$\mathsf{Even}$ take turns playing conditions $p_\alpha\in\bb{P}$ in $\leq_{\bb{P}}$-descending order. $\mathsf{Odd}$ 
		plays at all odd stages $\alpha$ and $\mathsf{Even}$ plays at all even stages $\alpha$, including the limit ones. The game begins with $\mathsf{Even}$'s 
		play of $p_0 = 1_{\bb{P}}$. If in the course of play a limit ordinal $\beta < \delta$ 
		is reached such that $\langle p_\alpha \mid \alpha < \beta \rangle$ has no lower bound in $\bb{P}$ then 
		$\mathsf{Odd}$ wins. Otherwise $\mathsf{Even}$ wins.
		\item We say that $\bb{P}$ is \emph{$\delta$-strategically closed} if $\mathsf{Even}$ has a winning strategy 
		in $\Game_\delta(\bb{P})$.
	\end{enumerate}
\end{definition}

If $\kappa$ is a cardinal and $\bb{P}$ is $(\kappa+1)$-strategically closed then forcing 
with $\bb{P}$ cannot add any new $\kappa$-sequences of elements from $\mathrm{V}$.

We are now ready to introduce the basic forcing notion to add non-$n$-trivial  $n$-coherent families of functions.

\begin{definition}
	Suppose that $n$ is a positive integer, $A$ is a nontrivial abelian group, and $\lambda \geq \aleph_n$ is a regular cardinal. 
	Let $\bb{P}(n, \lambda, A)$ denote the following forcing notion:
	\begin{itemize}
		\item the conditions of $\bb{P}(n, \lambda, A)$ are all $n$-coherent families of the form 
		$$p = \{\varphi^p_{\vec{\beta}}:\beta_0 \rightarrow A \mid \vec{\beta} \in [\delta_p + 1]^n\}$$ 
		where $\delta_p < \lambda$;
		\item if $p,q \in \bb{P}(n, \lambda, A)$, then $q \leq p$ if 
		\begin{itemize}
			\item $\delta_q \geq \delta_p$
			\item $\varphi^q_{\vec{\beta}} = \varphi^p_{\vec{\beta}}$ for all $\vec{\beta} \in [\delta_p+1]^n$.
		\end{itemize}
	\end{itemize}
\end{definition}

Fix now an integer $n \geq 2$ (the case of $n=1$ is similar and easier), 
a nontrivial abelian group $A$, and a regular cardinal $\lambda \geq \aleph_n$.

\begin{lemma} \label{strategic_lemma}
	The forcing $\bb{P}(n, \lambda, A)$ is $\lambda$-strategically closed and therefore preserves all 
	cardinalities and cofinalities less than or equal to $\lambda$.
\end{lemma}

\begin{proof}
	We describe a winning strategy for $\mathsf{Even}$ in $\Game_\lambda(\bb{P}(n, \lambda, A))$. 
	At stage $\eta$ of a game, $\mathrm{Ev}(\eta)$ will denote the set $\{\delta_{p_i} \mid i < \eta \text{ and } 
	i \mbox{ is even}\}$. For $\xi < \sup(\mathrm{Ev}(\eta))$, 
	let $\alpha_\xi = \min(\mathrm{Ev}(\eta) \setminus \xi + 1)$. Since the 
	sequence of played conditions must decrease, for any $\vec{\beta} \in [\cup_{i < \eta} \delta_{p_i} + 1]^n$ the notation $\varphi_{\vec{\beta}}$ is unambiguous, so we
	henceforth omit the superscripts recording the condition in which the function $\varphi_{\vec{\beta}}$ appears. $\mathsf{Even}$'s winning strategy consists in maintaining the following condition at all stages 
	$\eta < \lambda$:
	\begin{align}\label{winstrat3} 
		\tag{$\dagger(\eta)$} \textnormal{for all }\vec{\beta}\in[\mathrm{Ev}(\eta)]^{n+1},\hspace{0.1 cm}\sum_{j=0}^n(-1)^j\varphi_{\vec{\beta}^j}=0.
	\end{align}
	$\mathsf{Even}$ arranges this as follows. Suppose that $\eta < \lambda$ is even and that $(\dagger(\eta))$ holds for the sequence of plays $\langle p_i \mid i < \eta \rangle$. Let 
	$\delta = \sup\{\delta_{p_i} + 1 \mid i < \eta\}$. $\mathsf{Even}$ will play a condition $p_\eta$ 
	for which $\delta_{p_\eta} = \delta$; hence to define $p_\eta$ it will suffice to define 
	$\varphi_{\vec{\beta}\delta}$ for all $\vec{\beta} \in [\delta]^{n-1}$. We first define 
	$\varphi_{\vec{\beta}\delta}$ for $\vec{\beta} \in [\mathrm{Ev}(\eta)]^{n-1}$: for ease of notation, 
	adopt the convention that $\varphi_{\vec{\gamma}} = 0$ if any of the terms of the $n$-tuple $\vec{\gamma}$ repeat and let 
	
	$$\varphi_{\vec{\beta}\delta}(\xi) = (-1)^{n-1}\varphi_{\alpha_\xi\vec{\beta}}(\xi)$$
	for any $\vec{\beta} \in [\mathrm{Ev}(\eta)]^{n-1}$ and $\xi < \beta_0$. (In particular, 
	$\varphi_{\vec{\beta}\delta}(\xi) = 0$ if $\alpha_\xi = \beta_0$.) We check that this
	partial definition of $p_\eta$ achieves $(\dagger(\eta + 1))$: for such a $\vec{\beta}$ and $\xi$,
	\begin{align*}
		\sum_{j=0}^n(-1)^j\varphi_{(\vec{\beta}\delta)^j}(\xi) &= (-1)^n\varphi_{\vec{\beta}}(\xi) + \sum_{j=0}^{n-1} (-1)^j \varphi_{\vec{\beta}^j\delta} (\xi) \\ 
		&= (-1)^n \varphi_{\vec{\beta}}(\xi) + (-1)^{n-1} \sum_{j=0}^{n-1} (-1)^j \varphi_{\alpha_\xi\vec{\beta}^j}(\xi) 
	\end{align*}
	If $\alpha_\xi = \beta_0$ then this expression reduces to $(-1)^n \varphi_{\vec{\beta}}(\xi) + (-1)^{n-1} \varphi_{\vec{\beta}}(\xi) = 0$. 
	If $\alpha_\xi < \beta_0$ then it reduces to 
	\[
		(-1)^n\sum_{j=0}^n \varphi_{(\alpha_\xi\vec{\beta})^j}(\xi),
	\]
	which is $0$ by the assumption $(\dagger(\eta))$. In consequence, $\{\varphi_{\vec{\beta}\delta} \mid \vec{\beta} \in [\mathrm{Ev}(\eta)]^{n-1}\} \cup \bigcup_{\alpha < \eta} p_\alpha$ is an $n$-coherent family 
	witnessing $(\dagger(\eta + 1))$. $\mathsf{Even}$'s winning strategy is thus to play any $n$-coherent extension 
	\[
		p_\eta = \{\varphi_{\vec{\beta}\delta} \mid \vec{\beta} \in [\delta]^{n-1}\} \cup \bigcup_{\alpha < \eta} p_\alpha
	\]
	of this family at stage $\eta$. Such an extension exists by the higher-dimensional versions of Observation~\ref{transformation}.
\end{proof}

\begin{lemma} \label{nontrivial_generic_lemma}
	Suppose that $G$ is a $\bb{P}(n, \lambda, A)$-generic filter over $\mathrm{V}$ and 
	$\kappa < \lambda$ is an infinite regular cardinal for which there exists a height-$\kappa$ $A$-valued non-$(n-1)$-trivial $(n-1)$-coherent family of functions. Then 
	$\bigcup G$ is a non-$n$-trivial $n$-coherent family of functions $\Phi = \{\varphi_{\vec{\beta}} 
	: \beta_0 \rightarrow A \mid \vec{\beta} \in [\lambda]^n\}$.
\end{lemma}

\begin{proof}
	Let us assume that $n > 1$. The case $n=1$ is similar and easier.
	For all $\beta < \lambda$ the argument of Lemma~\ref{strategic_lemma} shows that the set $D_\beta := \{p \mid \delta_p > \beta\}$ 
	is dense in $\bb{P}(n, \lambda, A)$. Hence $\bigcup G$ is an $n$-coherent family of the form 
	$\Phi = \{\varphi_{\vec{\beta}} \mid \vec{\beta} \in [\lambda]^n\}$. Suppose for the sake of contradiction 
	that $\Phi$ is $n$-trivial and fix a name $\dot{\Psi} = \{\dot{\psi}_{\vec{\beta}} \mid \vec{\beta} \in [\lambda]^{n-1}\}$ 
	and a condition $p \in G$ such that $p \Vdash ``\dot{\Psi} \text{ $n$-trivializes } \dot{\Phi}".$

	Work now in $V$. Use $\mathsf{Even}$'s winning strategy in $\Game_\lambda(\bb{P}(n, \lambda, A))$ to 
	define a strictly decreasing sequence $\langle p_i \mid i < \kappa \rangle$ of conditions from 
	$\bb{P}(n, \lambda, A)$ with $p_0 = p$ such that for all $i < \kappa$ and all $\vec{\beta} \in 
	[\delta_{p_i}]^{n-1}$ the condition $p_{i+1}$ decides the value of $\dot{\psi}_{\vec{\beta}}$ to be some 
	$\psi_{\vec{\beta}}\in V$. Let $\delta = \sup\{\delta_{p_i} \mid i < \kappa\}$. 

	By assumption (together with Lemma~\ref{iff}), there exists a 
	non-$(n-1)$-trivial $(n-1)$-coherent family $\Upsilon = \{\upsilon_{\vec{\beta}}:\beta_0 \rightarrow A \mid 
	\vec{\beta} \in [\delta]^{n-1}\}$. For all $\vec{\beta} \in [\delta]^{n-1}$ define 
	$\varphi_{\vec{\beta}\delta}:\beta_0 \rightarrow A$ by letting $\varphi_{\vec{\beta}\delta} = 
	\upsilon_{\vec{\beta}} + (-1)^{n+1}\psi_{\vec{\beta}}$, and let 
	\[
		q = \{\varphi_{\vec{\beta}\delta} \mid \vec{\beta} \in [\delta]^{n-1}\} \cup \bigcup_{i < \kappa} p_i
	\]

	\begin{claim}
		$q$ is $n$-coherent and therefore a lower bound for $\langle p_i \mid i < \kappa \rangle$ in 
		$\bb{P}(n, \lambda, A)$.
	\end{claim}

	\begin{proof}
		Each $p_i$ is $n$-coherent; hence it suffices to check for $n$-coherence on index-sets of the form 
		$\vec{\beta} \cup \{\delta\}$, where $\vec{\beta} \in [\delta]^n$. In these cases, 
		\begin{align*}
			\sum_{j=0}^n(-1)^j \varphi_{(\vec{\beta}\delta)^j} &= 
			(-1)^n \varphi_{\vec{\beta}} + \sum_{j=0}^{n-1}(-1)^j\varphi_{\vec{\beta}^j\delta}  \\ 
			&= (-1)^n\left(\varphi_{\vec{\beta}} - \sum_{j=0}^{n-1} \psi_{\vec{\beta}^j}\right) 
			+ \sum_{j=0}^{n-1} \upsilon_{\vec{\beta}^j} =^* 0.
		\end{align*}
	\end{proof}

	Now take an $r\leq q$ such that for every $\beta \in [\delta]^{n-2}$ the condition $r$ decides the value of 
	$\dot{\psi}_{\vec{\beta}\delta}$ to be some $\psi^*_{\vec{\beta}}:\beta_0 \rightarrow A$ in $V$. Since $r$ 
	extends $p$ and therefore forces that $\dot{\Psi}$ $n$-trivializes $\dot{\Phi}$, it follows that for 
	every $\vec{\beta} \in [\delta]^{n-1}$
	\begin{align*}
		\sum_{j=0}^{n-2}\psi^*_{\vec{\beta}^j} =^* (-1)^n\psi_{\vec{\beta}} + \varphi_{\vec{\beta}\delta} 
		= (-1)^n\psi_{\vec{\beta}} + \upsilon_{\vec{\beta}} + (-1)^{n+1}\psi_{\vec{\beta}} = \upsilon_{\vec{\beta}}.
	\end{align*}
	The family $\Psi^* = \{\psi^*_{\vec{\beta}} \mid \vec{\beta} \in [\delta]^{n-2}\}$ in other words
	$(n-1)$-trivializes $\Upsilon$, contradicting our assumption that $\Upsilon$ is non-$(n-1)$-trivial.
\end{proof}

\begin{remark}
	For every nontrivial abelian group $A$ there exists a non-$0$-trivial $0$-coherent function 
	$\psi:\omega \rightarrow A$ (namely, any such function with infinite support). It follows that 
	for every abelian group $A$ and every regular uncountable $\lambda$, if $G$ is $\bb{P}(1, \lambda, A)$-generic
	over $V$, then $\bigcup G$ is nontrivial coherent. 

	Similarly, by Corollary~\ref{DA}, for every nontrivial abelian group $A$ there exists an $A$-valued non-$1$-trivial $1$-coherent
	family of functions indexed by $\omega_1$. It follows that for every 
	such $A$ and regular cardinal $\lambda > \aleph_1$, if $G$ is $\bb{P}(2, \lambda, A)$-generic 
	over $V$ then $\bigcup G$ is a height-$\lambda$ $A$-valued non-$2$-trivial $2$-coherent family of functions.

	Finally, by Theorem~\ref{nonntrivncohomegan}, for every $2 \leq n < \omega$ there is an abelian group $A_n$ 
	and an $A_n$-valued non-$n$-trivial $n$-coherent family of functions indexed by $[\omega_n]^n$. 
	It follows that for every $n \geq 2$ and regular $\lambda > \aleph_n$, if 
	$G$ is $\bb{P}(n+1, \lambda, A_n)$-generic over $V$, then $\bigcup G$ is height-$\lambda$ $A_n$-valued non-$n$-trivial $n$-coherent family of functions. Alternately, for any nontrivial abelian group $A$ and regular $\lambda\geq\aleph_n$, one may add a height-$\lambda$ $A$-valued non-$n$-trivial $n$-coherent family of functions via repeated applications of Lemma \ref{nontrivial_generic_lemma}.
	
	Such constructions, along with those of the previous subsection, are a partial heuristic for the $\mathsf{ZFC}$ theorems \ref{goblot} and \ref{nonntrivncohomegan} of Section \ref{I.5}: the ``space'' of repeated ``jumps'' in cardinality is needed in order for higher-dimensional nontrivial coherence relations to appear.
\end{remark}

Once a non-$n$-trivial $n$-coherent family has been added by forcing, there is a natural further forcing 
notion to $n$-trivialize it. Our definition will look slightly different depending on whether or not $n = 1$.

\begin{definition}
	Suppose that $\lambda$ is a regular uncountable cardinal and $\Phi = \{\varphi_\beta:\beta 
	\rightarrow A \mid \beta \in \lambda\}$ is a non-$1$-trivial $1$-coherent family. Then 
	$\bb{T}(\Phi)$ is the following forcing notion:
	\begin{itemize}
		\item The conditions of $\bb{T}(\Phi)$ are all functions $t:\delta_t \rightarrow A$ such that 
		$\delta_t < \lambda$ and $t =^* \varphi_{\delta_t}$.
		\item If $t,r \in \bb{T}(\Phi)$ then $r \leq t$ if $\delta_r \geq \delta_t$ and 
		$r \restriction \delta_t = t$.
	\end{itemize}
\end{definition}

\begin{definition}
	Suppose that $n \geq 2$ and $\lambda$ is a regular uncountable cardinal  and $\Phi = \{\varphi_{\vec{\beta}}:
	\beta_0 \rightarrow A \mid \vec{\beta} \in [\lambda]^n\}$ is a non-$n$-trivial $n$-coherent family. 
	Then $\bb{T}(\Phi)$ is the following forcing notion:
	\begin{itemize}
		\item The conditions of $\bb{T}(\Phi)$ are all families $t = \{\psi^t_{\vec{\alpha}}:\alpha_0 \rightarrow A 
		\mid \vec{\alpha} \in [\delta_t]^{n-1}\}$ such that for all $\vec{\beta} \in [\delta_t]^n$ 
		\[
			\sum_{j=0}^{n-1}(-1)^j \psi^t_{\vec{\beta}^j} =^* \varphi_{\vec{\beta}}.
		\]
		\item If $t, r \in \bb{T}(\Phi)$ then $r \leq t$ if $\delta_r \geq \delta_t$ and $\psi^r_{\vec{\alpha}} = \psi^t_{\vec{\alpha}}$ for all 
		$\vec{\alpha} \in [\delta_t]^{n-1}$.
	\end{itemize}
\end{definition}

For $\bb{P}(n, \lambda, A)$-generic families of functions $\Phi=\cup\, G$, the forcing $\bb{T}(\Phi)$ is particularly well-behaved, in senses the following lemma makes precise. This lemma, a modification of arguments in Kunen's in \cite[\S 3]{kunen_saturated_ideals}, will play a major role in the next subsection.

\begin{lemma} \label{dense_closed_lemma}
	Suppose that $n$ is a positive integer, $A$ is an abelian group, and $\lambda \geq \aleph_n$ is a regular cardinal.
	Let $\bb{P} = \bb{P}(n, \lambda, A)$ and let $\dot{\Phi} = \{\dot{\varphi}_{\vec{\alpha}} \mid \vec{\alpha} \in [\lambda]^n\}$ 
	be a canonical $\bb{P}$-name for the union of the $\bb{P}$-generic filter. Let $\dot{\bb{T}}$ be 
	a canonical $\bb{P}$-name for $\bb{T}(\dot{\Phi})$. Then the following two statements hold.
	\begin{enumerate}
		\item $\bb{P} * \dot{\bb{T}}$ has a dense $\lambda$-directed closed subset of size $\lambda^{<\lambda}$.
		\item If $n = 1$ and $|A| < \lambda$, then $\dot{\bb{T}}$ is forced by the empty condition to have the 
			$\lambda$-cc in $V^{\bb{P}}$.
	\end{enumerate}
\end{lemma}

\begin{proof}
	(1) Assume for ease of notation that $n \geq 2$. The case in which $n=1$ is similar and easier.
	Let $\bb{U}$ be the set of conditions $(p, \dot{t}) \in \bb{P} * \dot{\bb{T}}$ such that
	\begin{itemize}
		\item there is a $t \in \mathrm{V}$ such that $\dot{t} = \check{t}$;
		\item $\delta_p = \delta_t$;
		\item $\varphi_{\vec{\alpha}\delta_p} = (-1)^{n+1}\psi^t_{\vec{\alpha}}$ for all $\vec{\alpha} \in [\delta_p]^{n-1}$.
	\end{itemize}
	Clearly $|\bb{U}| = \lambda^{<\lambda}$.
	We claim that $\bb{U}$ is a dense $\lambda$-closed subset of $\bb{P} * \dot{\bb{T}}$. We first argue its 
	density in $\bb{P} * \dot{\bb{T}}$. To that end, fix $(p_0, \dot{t}_0) \in \bb{P} * \dot{\bb{T}}$. Since $\bb{P}$ is $\lambda$-strategically 
	closed, we can assume by extending $p_0$ if necessary that there is 
	a $t_0 \in \mathrm{V}$ such that $p_0 \Vdash_{\bb{P}}``\dot{t}_0 = \check{t}_0"$ and also that $\delta_{p_0} > \delta_{t_0}$.
	We then extend $p_0$ to a condition $p$ with $\delta_p = \delta_{p_0} + 1$; as in the proof of Lemma \ref{strategic_lemma} it suffices 
	to specify $\varphi^p_{\vec{\alpha}\delta_p}$ for all $\vec{\alpha} \in [\delta_p]^{n-1}$. This task breaks into three cases:

	\textbf{Case 1: $\vec{\alpha} \in [\delta_{t_0}]^{n-1}$.} In this case, let $\varphi^p_{\vec{\alpha}\delta_p} 
	= (-1)^{n+1} \psi^t_{\vec{\alpha}}$.

	\textbf{Case 2: $\vec{\alpha} \in [\delta_{p_0}]^{n-1}$ and $\alpha_{n-2} \geq \delta_{t_0}$.} In this case, 
	let $\varphi^p_{\vec{\alpha}\delta_p} = \varphi^{p_0}_{\vec{\alpha}\delta_{p_0}}$.

	\textbf{Case 3: $\alpha_{n-2} = \delta_{p_0}$.} In this case, let $\varphi^p_{\vec{\alpha}\delta_p} = 0$.

	Using the fact that $(p_0, \dot{t}_0) \in \bb{P} * \dot{\bb{T}}$ and hence that 
	$\{(-1)^{n+1} \psi^t_{\vec{\alpha}} \mid \vec{\alpha} \in [\delta_{t_0}]^{n-1}\}$ $n$-trivializes 
	$\{\varphi^{p_0}_{\vec{\beta}} \mid \vec{\beta} \in [\delta_{t_0}]^n\}$, it is straightforward to 
	verify that $p$ is $n$-coherent. Define 
	$t = \{\psi^t_{\vec{\alpha}} \mid \vec{\alpha} \in [\delta_p]^{n-1}\}$ by letting 
	$\psi^t_{\vec{\alpha}} = (-1)^{n+1} \varphi^p_{\vec{\alpha}\delta_p}$ for all 
	$\vec{\alpha} \in [\delta_p]^{n-1}$, and let $\dot{t} = \check{t}$. It is now routine to 
	verify that $(p, \dot{t})$ extends $(p_0, \dot{t}_0)$ and is in $\bb{U}$. It follows that 
	$\bb{U}$ is dense in $\bb{P} * \dot{\bb{T}}$.

	We now show that $\bb{U}$ is $\lambda$-directed closed. First note that $\bb{U}$ is \emph{tree-like}, 
	i.e., if $u,v,w \in \bb{U}$ and $w \leq_{\bb{U}} u,v$, then $u$ and $v$ are $\leq_\bb{U}$-comparable. 
	It thus suffices to show that $\bb{U}$ is $\lambda$-closed. To this end, let $\eta < \lambda$ 
	be a limit ordinal and let $\langle (p_\xi, \dot{t}_\xi) \mid \xi < \eta \rangle$ be a 
	strictly decreasing sequence of conditions in $\bb{U}$. Let $\delta^* = \sup\{\delta_{p_\xi} \mid 
	\xi < \eta\}$. We will define a condition $p^* \in \bb{P}$ with $\delta_{p^*} = \delta^*$ 
	such that $p^*$ is a lower bound for $\langle p_\xi \mid \xi < \eta \rangle$ by defining, as usual, $\varphi^{p^*}_{\vec{\alpha}\delta^*}$ for all 
	$\vec{\alpha} \in [\delta^*]^{n-1}$. To this end, fix such a $\vec{\alpha}$ and a $\xi < \eta$ such that $\alpha_{n-2} < \delta_\xi$ and let 
	$\varphi^{p^*}_{\vec{\alpha}\delta^*} = \varphi^{p_\xi}_{\vec{\alpha}\delta_\xi}$. (Note that this 
	definition is independent of our choice of $\xi$.) Finally, let $t^* = \bigcup_{\xi < \eta} 
	t_\xi$ and let $\dot{t}^* = \check{t}^*$. The condition $(p^*, \dot{t}^*)$ is a lower bound for 
	$\langle (p_\xi, \dot{t}_\xi) \mid \xi < \eta \rangle$ in $\bb{U}$, hence 
	$\bb{U}$ is $\lambda$-closed.

	(2) Consider now the case of $n = 1$ and $|A|<\lambda$. Fix a condition $p \in \bb{P}$ and a $\bb{P}$-name $\dot{X}$ 
	for a maximal antichain in $\dot{\bb{T}}$. We will extend $p$ to a condition $p^*$ such that $p^* \Vdash_{\bb{P}}``|\dot{X}| < \check{\lambda}"$.
	Enumerate the collection of all finitely-supported functions from $\lambda$ to $A$ as $\langle f_\alpha \mid \alpha < \lambda \rangle$. 
	This collection guides the recursively construction of a strictly decreasing sequence of conditions $\langle p_\alpha \mid \alpha < \lambda \rangle$ in 
	$\bb{P}$ as follows. Throughout the construction, we maintain the following two hypotheses:
	\begin{itemize}
		\item $\mathrm{supp}(f_\alpha) \subseteq \delta_{p_\alpha}$ for all $\alpha < \lambda$, and
	 	\item $\varphi^{p_\beta}_{\delta_{p_\beta}} \restriction \delta_{p_\alpha} 
	 	= \varphi^{p_\alpha}_{\delta_{p_\alpha}}$ for all $\alpha < \beta < \lambda$. 
	\end{itemize}
	To begin, let $p_0$ be any extension of $p$ such that $\mathrm{supp}(f_0) \subseteq \delta_{p_0}$. Next, at any successor stage $\alpha+1<\lambda$ of the construction, $p_\alpha$ has been defined; let $t_\alpha$ denote the function $\varphi^{p_\alpha}_{\delta_{p_\alpha}} + f_\alpha$. Observe that 
	$p_\alpha \Vdash_{\bb{P}}``\check{t_\alpha} \in \dot{\bb{T}}"$. Since $p_\alpha$ also forces 
	that $\dot{X}$ is a maximal antichain in $\dot{\bb{T}}$, there exists a condition $p^*_{\alpha + 1} \leq_{\bb{P}} p_\alpha$
	and a function $t^*_\alpha:\delta^*_\alpha \rightarrow A$ such that 
	\begin{itemize}
	  \item either $t^*_\alpha$ end-extends $t_\alpha$ or vice versa, and
	  \item $p^*_{\alpha + 1} \Vdash_{\bb{P}}``\check{t}^*_\alpha \in \dot{X}"$.
	\end{itemize}
	Now let $p_{\alpha + 1}$ be any extension of $p^*_{\alpha + 1}$ such that
	\begin{itemize}
		\item $(\mathrm{supp}(f_{\alpha + 1}) \cup \delta^*_\alpha) \subseteq \delta_{p_{\alpha + 1}}$;
		\item $\varphi^{p_{\alpha + 1}}_{\delta_{p_{\alpha + 1}}} \restriction \delta_{p_\alpha} = \varphi^{p_\alpha}_{\delta_{p_\alpha}}$;
		\item if $\delta^*_\alpha > \delta_{p_\alpha}$, then $\varphi^{p_{\alpha + 1}}_{\delta_{p_{\alpha + 1}}} 
			\restriction [\delta_{p_\alpha}, \delta^*_\alpha) = t^*_\alpha \restriction [\delta_{p_\alpha}, \delta^*_\alpha)$.
	\end{itemize}
	The essential point here is that $(\varphi^{p_{\alpha + 1}}_{\delta_{p_{\alpha + 1}}} + f_\alpha) 
	\restriction \delta^*_\alpha = t^*_\alpha$.

	Finally, suppose that $\beta < \lambda$ is a limit ordinal and that $\langle p_\alpha \mid \alpha < \beta \rangle$ 
	has been constructed. Let $\delta^*_\beta = \sup\{\delta_{p_\alpha} \mid \alpha < \beta\}$. By our hypotheses, letting $\varphi^{p^*_\beta}_{\delta^*_\beta} = \bigcup_{\alpha < \beta} \varphi^{p_\alpha}_{\delta_{p_\alpha}}$ defines a lower bound $p_\beta^*$ 
	for $\langle p_\alpha \mid \alpha < \beta \rangle$, for which $\delta_{p^*_\beta} = \delta^*_\beta$.
	Let $p_\beta$ be any extension of $p^*_\beta$ for which $\mathrm{supp}(f_\beta) \subseteq \delta_{p_\beta}$ 
	and $\varphi^{p_\beta}_{\delta_{p_\beta}} \restriction \delta^*_\beta = \varphi^{p^*_\beta}_{\delta^*_\beta}$.

	Observe that the ordinals $\{\delta_\beta^* \mid \beta \in \mathrm{Lim}(\lambda)\}$ of the above construction form a club subset of $\lambda$. Therefore there exists a 
	$\beta \in \mathrm{Lim}(\lambda)\}$ such that
	\begin{itemize}
		\item $\delta_\beta^* = \beta$;
		\item $\{f_\alpha \mid \alpha < \beta\}$ enumerates all finitely supported functions from 
			$\lambda$ to $A$ with support contained in $\beta$.
	\end{itemize}
	We claim that the condition $p_\beta^*$ defined at stage $\beta$ of the above construction is the $p^*\leq p$ we had desired. In particular, we claim that $p_\beta^*$ forces every element 
	of $\dot{X}$ to have domain contained in $\beta$. Since every 
	element of $\dot{X}$ must then be a restriction of a finite modification of $\varphi^{p_\beta^*}_\beta$, 
	the condition $p_\beta^*$ will force $``|\dot{X}| \leq \max\{|\beta|, |A|\} < \check{\lambda}$"; this will conclude the proof.

	To see this, suppose towards contradiction that there exists a $q \leq_{\bb{P}} p_\beta^*$ and 
	$t$ such that $\mathrm{dom}(t) \nsubseteq \beta$ and $q \Vdash_{\bb{P}}``\check{t} \in \dot{X}"$.
	Let $g = t \!\restriction \!\beta - \varphi^{p^*_\beta}_\beta$, and let $f:\lambda \rightarrow A$ be the trivial extension of 
	$g$, so that $\mathrm{supp}(f) = \mathrm{supp}(g)$. Since the finite support of $f$ is contained in $\beta$, there exists some $\alpha < \beta$ for which $f = f_\alpha$. Consider now 
	the function $t_\alpha^*$ constructed at stage $\alpha + 1$ of the construction. Unwinding 
	definitions, we find both that $t_\alpha^* = t\! \restriction\! \delta_\alpha^*$ and that 
	$p_\beta^* \Vdash_{\bb{P}}``\dot{X} \text{ is an antichain and } \check{t}, \check{t_\alpha^*} \in \dot{X}"$. This is the desired contradiction.
\end{proof}

\subsection{Separating compactness principles}\label{II.5}

We conclude this section with two results separating the assertion that $\chm^1(\lambda, A) = 0$ 
for all abelian groups $A$ from other compactness principles at $\lambda$. We first show that, 
although by Theorem \ref{todorcevicsquare} the aforementioned assertion implies the failure of $\square(\lambda)$, it does not imply the failure of the weaker principle $\square(\lambda, 2)$. 
We then show that it is not implied by stationary reflection at $\lambda$. Since the proofs of 
these results are minor variants of proofs appearing elsewhere, we mainly sketch their arguments, providing references for readers wanting more detail. We begin with the relevant definitions.

\begin{definition}
	Suppose that $\lambda$ is a regular uncountable cardinal and $\kappa$ is a cardinal with $1 \leq \kappa < \lambda$. 
	The principle $\square(\lambda, \kappa)$ 
	asserts the existence of a sequence $\vec{\mathcal{C}} = \langle \mathcal{C}_\alpha \mid \alpha < \lambda \rangle$ 
	such that 
	\begin{enumerate}
		\item for every limit ordinal $\alpha < \lambda$, $\mathcal{C}_\alpha$ is a collection of 
		clubs in $\alpha$ such that $1 \leq |\mathcal{C}_\alpha| \leq \kappa$;
		\item for all limit ordinals $\alpha < \beta < \lambda$ and all $C \in \mathcal{C}_\beta$, 
		if $\alpha \in \mathrm{Lim}(C)$, then $C \cap \alpha \in \mathcal{C}_\alpha$;
		\item there is no club $D$ in $\lambda$ such that 
		$D \cap \alpha \in \mathcal{C}_\alpha$ for every $\alpha \in \mathrm{Lim}(D)$.
	\end{enumerate}
\end{definition}

Note that $\square(\lambda) = \square(\lambda, 1)$ and that if $1 \leq \kappa_0 < \kappa_1 < \lambda$ then $\square(\lambda, \kappa_0)$ implies 
$\square(\lambda, \kappa_1)$. Though the weakening of $\square(\lambda)$ to $\square(\lambda, 2)$ may appear innocuous, we will see momentarily that these two principles carry significantly different implications for the vanishing of $\chm^1(\lambda, \mathcal{A}_d)$.

\begin{definition}
	Suppose that $\lambda$ is a regular uncountable cardinal, $S$ is a stationary subset of $\lambda$, 
	and $\beta < \lambda$ has uncountable cofinality.
	We say that $S$ \emph{reflects at $\beta$} if $S \cap \beta$ is stationary in $\beta$. We simply say 
	that $S$ \emph{reflects} if there is $\beta < \lambda$ of uncountable cofinality such that 
	$S$ reflects at $\beta$.
\end{definition}

The failure of $\square(\lambda)$, the failure of $\square(\lambda, 2)$, the assertion that 
every stationary subset of $\lambda$ reflects, and the triviality of $\chm^1(\lambda, A)$ 
are clearly all assertions of compactness at $\lambda$. The point of our last two results 
is to show that these are all essentially different principles. 

\begin{theorem} \label{square_separation_thm}
	Suppose that $\lambda$ is weakly compact. Then there is a forcing extension in which
	\begin{enumerate}
		\item $\lambda$ remains inaccessible;
		\item $\square(\lambda, 2)$ holds;
		\item $\chm^1(\lambda, \mathcal{A}_d) = 0$ for every abelian group $A$.
	\end{enumerate}
\end{theorem} 

\begin{proof}
	By way of a preliminary forcing, we may assume that the weak compactness of 
	$\lambda$ is indestructible by any forcing $\mathbb{P}$ such that $\mathbb{P}$ is $\lambda$-directed closed and $|\mathbb{P}|\leq\lambda$ (this observation is due to Silver; see \cite[Example 16.2]{cummings_handbook} for 
	details). Let $\bb{P}$ be the standard forcing to add a $\square(\lambda, 2)$-sequence 
	by initial segments (see \cite[Definition 3.6]{hayut_lh}) and let 
	$G$ be $\bb{P}$-generic over $V$. The following is shown in \cite[Lemma 3.2 and Corollary 3.2]{hayut_lh}:
	\begin{enumerate}
		\item[(i)] $\square(\lambda, 2)$ holds in $V[G]$;
		\item[(ii)] There is a $\bb{P}$-name $\dot{\bb{T}}$ for a poset such that, in $V$, 
		both $\bb{P} * \dot{\bb{T}}$ and $\bb{P} * \dot{\bb{T}}^2$ have dense 
		$\lambda$-directed closed subsets of size $\lambda$.
	\end{enumerate}
	$V[G]$ will be our desired model. By item (ii) above, forcing with $\bb{P}$ preserves 
	all cofinalities less than or equal to $\lambda$, so $\lambda$ remains inaccessible in 
	$V[G]$. Hence conclusions (1) and (2) of the theorem are clear. 

	It remains to show that $\chm^1(\lambda, \mathcal{A}_d) = 0$ for every abelian group $A$ in $V[G]$.
	Suppose to the contrary that
	$\Phi$ is 
	a height-$\lambda$ $A$-valued nontrivial coherent family in $V[G]$. By item (ii) together with our 
	indestructibility assumptions, forcing over $V[G]$ with either $\bb{T}$ or $\bb{T} \times \bb{T}$ resurrects the weak compactness of $\lambda$. Hence by Theorem~\ref{nowklycpt} 
	and Remark~\ref{weakly_compact_remark}, $\Phi$ becomes trivial after forcing with $\bb{T}$. However, 
	since $\bb{T} \times \bb{T}$ preserves the uncountable cofinality of $\lambda$,
	Lemma~\ref{one_branch_lemma} implies that $\Phi$ remains nontrivial after 
	forcing with $\bb{T}$. This contradiction completes the proof.
\end{proof}

\begin{theorem}\label{stationaryreflectiontheorem}
	Suppose that $\lambda$ is weakly compact. Then there is a forcing extension in which 
	\begin{enumerate}
		\item $\lambda$ remains inaccessible;
		\item $\chm^1(\lambda, \mathcal{A}_d) \neq 0$ for every abelian group $A$;
		\item for every collection $\bar{S} = \{S_\eta \mid \eta < \lambda\}$ of stationary subsets of 
		$\lambda$, the set 
		\[
			T_{\bar{S}} = \{\beta < \lambda \mid \mathrm{cf}(\beta) > \omega \text{ and, for all }\alpha < \beta, ~ 
			S_\alpha \text{ reflects at } \beta\}
		\]
		is stationary in $\lambda$;
		\item $\square(\lambda, \kappa)$ fails for all $\kappa < \lambda$.
	\end{enumerate}
\end{theorem}

\begin{proof}
	As in the proof of Theorem~\ref{square_separation_thm}, we may assume that the weak compactness 
	of $\lambda$ is indestructible under $\lambda$-directed closed forcings of size $\lambda$. 
	Let $\bb{P} = \bb{P}(1, \lambda, \bb{Z}_2)$ 
	be the forcing poset described above for adding a height-$\lambda$ nontrivial coherent family of $\mathbb{Z}_2$-valued functions
	$\Phi = \cup\, G$, where $G$ is a $\bb{P}$-generic ultrafilter over $V$. We claim that $V[G]$ is the forcing extension desired. By the arguments of the previous subsection,
	$\lambda$ will be inaccessible and $\Phi$ will be nontrivial coherent in $V[G]$. 
	It follows then from Corollary~\ref{DA} and Lemma~\ref{gddg} that $\chm^1(\lambda, \mathcal{A}_d) \neq 0$ for 
	every abelian group $A$ in $V[G]$.

	We next argue item (3) in $V[G]$. To this end, fix in $V[G]$ a collection 
	$\bar{S} = \{S_\eta \mid \eta < \lambda\}$ of stationary subsets of $\lambda$. Also in $V[G]$, let 
	$\bb{T} = \bb{T}(\Phi)$ be the forcing to trivialize $\Phi$, and let $H$ be $\bb{T}$-generic over $V[G]$.
	By Lemma~\ref{dense_closed_lemma} applied in $V$, the forcing $\bb{P} * \dot{\bb{T}}$ has a dense 
	$\lambda$-directed closed subset of size $\lambda^{<\lambda} = \lambda$. Hence by our indestructibility 
	assumptions, $\lambda$ is weakly compact in $V[G*H]$. Moreover, by Clause (2) of Lemma~\ref{dense_closed_lemma}, 
	$\bb{T}$ has the $\lambda$-cc in $V[G]$ and therefore preserves all stationary subsets of $\lambda$. 
	It follows that each $S_\eta$ $(\eta < \lambda)$ remains stationary in $V[G*H]$. A standard application of 
	$\Pi^1_1$-indescribability yields that, in $V[G*H]$, $T_{\bar{S}}$ is stationary in $\lambda$. To conclude, observe that
	$(V_\lambda)^{V[G]}=(V_\lambda)^{V[G*H]}$, hence $T_{\bar{S}}$ defined in $V[G]$ is the same as 
	$T_{\bar{S}}$ defined in $V[G*H]$ and so, since stationarity is downwards absolute, $T_{\bar{S}}$ is 
	stationary in $V[G]$.

	We lastly argue item (4). Suppose instead for contradiction that $\kappa < \lambda$ and 
	$\square(\lambda, \kappa)$ holds in $V[G]$. By \cite[Theorem 2.3]{hayut_lh}, it follows that there is 
	a collection $\bar{S}_0 = \{S_\eta \mid \eta < \kappa\}$ such that, for every $\beta < \lambda$ 
	of uncountable cofinality, there is $\eta < \kappa$ such that $S_\eta$ does not reflect at $\beta$. 
	Extend $\bar{S}_0$ to $\bar{S} = \{S_\eta \mid \eta < \lambda\}$ by letting $S_\eta = \lambda$ 
	for all $\eta > \kappa$. It follows from our choice of $\bar{S}_0$ that $T_{\bar{S}} \setminus \kappa 
	= \emptyset$, contradicting (3).
\end{proof}
In particular, comparing the case $\kappa=1$ of the above theorem with Theorem \ref{todorcevicsquare}, the existence of a $\square(\lambda)$ sequence is a strictly stronger assumption than the existence of a height-$\lambda$ nontrivial coherent family of functions.
\begin{remark}
	Using similar techniques, we could obtain models similar to those obtained in Theorems 
	\ref{square_separation_thm} and \ref{stationaryreflectiontheorem} in which, in the final model, 
	$\lambda$ is either the successor of a regular cardinal (e.g., $\lambda = \aleph_2$) or, 
	if in our initial model $\lambda$ is the successor of a singular limit of supercompact 
	cardinals, $\lambda$ is the successor of a singular cardinal (e.g., $\lambda = \aleph_{\omega + 1}$). 
	We refer interested readers to \cite{hayut_lh} for details regarding how to modify the 
	arguments of Theorems \ref{square_separation_thm} and \ref{stationaryreflectiontheorem} 
	to obtain such models.
\end{remark}

\section{Conclusion and questions}\hfill

We close with the following summary diagram. Boldfaced are $\mathsf{ZFC}$ phenomena: the vanishing and non-vanishing cohomology groups, respectively, of Theorem \ref{goblot}, Theorem \ref{nonntrivncohomegan}, and Example \ref{zerogroup}. Lightfaced are sensitivities to assumptions supplementary to the $\mathsf{ZFC}$ axioms: \begin{enumerate}
\item The vanishing or non-vanishing, in local or global senses, of $\chm^1(\kappa,\mathcal{A}_d)$ for regular cardinals $\kappa>\omega_1$ that are not weakly compact, by Theorem \ref{stronglycmpct}, Remark \ref{local}, Theorem \ref{pidtheorem}, and Theorem \ref{todorcevicsquare}. 
\item The consistent global non-vanishing of $\chm^n(\kappa,\mathcal{A}_d)$ for all $n\geq 2$ and regular cardinals $\kappa>\omega_n$ that are not weakly compact, by Corollary \ref{L}.
\end{enumerate}
How the pattern of known possibilities extends along the diagram's ellipses should at this point be plain. Note moreover that, by Lemma \ref{iff}, the columns of the diagram may be read as describing the cohomological behavior of the \emph{classes} of ordinals $\text{Cof}(\omega)$, $\text{Cof}(\omega_1)$, etcetera, so that the diagram in its ideal extension in fact charts the known cohomological possibilites for any ordinal $\xi$.
\begin{center}
\vspace{.5 cm}
\begin{tabular}{l | *{5}{c}r}

$\,\vdots\!$ & $\vdots$ & $\vdots$ & $\vdots$ & $\vdots$ & \reflectbox{$\ddots$} \\
$\chm^3$ & \textbf{0} & \textbf{0} & \textbf{0} & \textbf{nonzero} & \dots  \\
$\chm^2$         & \textbf{0} & \textbf{0} & \textbf{nonzero} & consistently nonzero & \dots  \\
$\chm^1$          & \textbf{0} & \textbf{nonzero} & \emph{independent} & \emph{independent} & \dots  \\
$\chm^0$     & \textbf{nonzero} & \textbf{nonzero} & \textbf{nonzero} & \textbf{nonzero} & \dots  \\
\hline
& $\omega$ & $\omega_1$ & $\omega_2$ & $\omega_3$ & \dots \\
\end{tabular}
\end{center}
\vspace{.5 cm}
Pictured, plainly, are phenomena of dimension; an evocation of $n$-spheres in the above diagram is strong. Conspicuous as well in the above is the case $n=2$ and $\lambda=\omega_3$  of the most immediate open question:

\begin{question}
	Suppose that $n > 1$ and that $\lambda>\omega_n$ is a regular cardinal that is not weakly compact. Is it consistent (relative to large cardinal assumptions) 
	that $\chm^n(\lambda, \mathcal{A}_d) = 0$ for every 
	abelian group $A$? More globally, is it consistent that $\chm^n(\lambda, \mathcal{A}_d) = 0$ for every 
	abelian group $A$ and every regular cardinal $\lambda > \omega_n$?
\end{question}

In perhaps the simplest scenario, the argument of Theorem \ref{stronglycmpct} would adapt wholesale, and the L\'{e}vy collapse of a strongly compact $\kappa$ to $\omega_3$, for example, would secure the vanishing of $\chm^2(\lambda,\mathcal{A}_d)=0$ for all abelian groups $A$ and cardinals $\lambda\geq \kappa$. Observe that this scenario would follow immediately from an affirmative answer to any of several variants of the following:

\begin{question}
	Is there a generalization of Lemma~\ref{one_branch_lemma} to higher-dimensional nontrivial coherent 
	families? More concretely, fix $n > 1$ and an ordinal $\delta$ of cofinality at least $\omega_n$, and let $\Phi$ be a height-$\delta$ non-$n$-trivial $n$-coherent 
	family and let $\bb{P}$ be a forcing poset such that $\Vdash_{\bb{P} \times \bb{P}}``\cf(\delta) \geq \omega_n"$. 
	Must $\Phi$ remain non-$n$-trivial after forcing with $\bb{P}$?
\end{question}
Like all of our material above, these questions underscore a broad and stark disparity between what we as set-theorists understand of what might be called ``two-dimensional'' or ``first cohomological'' incompactness principles, and what we can say of incompactness principles of higher order. Classical examples of the former are the existence of square or nontrivial coherent sequences, or of Aronsjazn trees; as noted, each of these principles essentially extrapolates the $\mathsf{ZFC}$ combinatorics of $\omega_1$. It appears likely, in conclusion, that what we do succeed in understanding of higher-dimensional incompactness principles will not be unconnected to what we succeed in understanding of the ordinals $\omega_n$ $(n>1)$ themselves.\\

{\bf Acknowledgements.} Portions of this material appeared in the first author's Ph.D. thesis; he would like to thank the members of his committee --- Justin Moore, Slawomir Solecki, and Jim West --- as well as Stevo Todorcevic, for all their support and instruction. He would like to thank Assaf Rinot as well, to whom he owes his awareness of the relative strengths of the principles $\square(\omega_2)$ and $\chm^1(\omega_2)\neq 0$, and of the relevance of \cite{kunen_saturated_ideals} to these considerations. 

\bibliographystyle{plain}
\bibliography{CohomologyOfOrdinalsI}
\end{document}